\documentclass[preprint,11pt]{elsarticle}
\usepackage{lineno,hyperref}
\modulolinenumbers[5]

\usepackage{amsmath}
\usepackage{amsfonts,amsthm,amssymb}
\usepackage{amsfonts}
\usepackage{graphics}
\usepackage{pstricks}
\usepackage{graphicx}
\usepackage{cleveref}
\usepackage{extarrows,chngpage,array,float}     
\usepackage{amssymb}
\usepackage{latexsym,bm}
\usepackage{graphicx}
\usepackage{subfigure}
\usepackage{caption}
\usepackage{float}
\usepackage{amsmath}
\usepackage{cases}
\usepackage{mathrsfs}
\usepackage{amsfonts}
\usepackage{color}
\usepackage[lined,boxed,commentsnumbered,ruled]{algorithm2e}
\usepackage{amssymb}
\usepackage{amsfonts,indentfirst}

\newtheorem{theorem}{Theorem}[section]
\newtheorem{lemma}[theorem]{Lemma}

\newtheorem{problem}[theorem]{Problem}

\allowdisplaybreaks   

\begin{document}

\begin{frontmatter}

\title{Extremal trees, unicyclic and bicyclic graphs with respect to $p$-Sombor spectral radii}

\author{ Ruiling Zheng, Tianlong Ma\footnote{Corresponding author}, Xian'an Jin\\
\small School of Mathematical Sciences\\[-0.8ex]
\small Xiamen University\\[-0.8ex]
\small Xiamen 361005, P. R. China\\
\small\tt Email: rlzheng2017@163.com, tianlongma@aliyun.com, xajin@xmu.edu.cn}

\begin{abstract}
For a graph $G=(V,E)$ and $v_{i}\in V$, denote by $d_{v_{i}}$ (or $d_{i}$ for short) the degree of vertex $v_{i}$.
The $p$-Sombor matrix $\textbf{S}_{\textbf{p}}(G)$ ($p\neq0$) of a graph $G$ is a square matrix, where the
$(i,j)$-entry is equal
to $\displaystyle (d_{i}^{p}+d_{j}^{p})^{\frac{1}{p}}$ if the vertices $v_{i}$ and $v_{j}$ are adjacent, and 0
otherwise. The $p$-Sombor spectral radius of $G$, denoted by $\displaystyle \rho(\textbf{S}_{\textbf{p}}(G))$,
is the largest eigenvalue of the
$p$-Sombor matrix $\textbf{S}_{\textbf{p}}(G)$. In this paper, we consider the extremal trees, unicyclic and bicyclic graphs with respect to the $p$-Sombor spectral radii. We characterize completely the extremal graphs with the first three maximum Sombor spectral radii, which answers partially a problem posed by Liu et al. in [MATCH Commun. Math. Comput. Chem. 87 (2022) 59-87].

\end{abstract}

\begin{keyword}
$p$-Sombor spectral radius; weighted adjacency matrix; tree; unicyclic graph; bicyclic graph
\MSC[2020] 05C09\sep 05C35\sep 05C50
\end{keyword}

\end{frontmatter}


\section{Introduction}
Graphs considered in this paper are all connected and simple, i.e., have no loops and parallel edges. Let $G=(V(G),E(G))$ be a graph with vertex set $V(G)=\{v_{1},v_{2},\ldots,v_{n}\}$ and edge set $E(G)$. An edge $e\in
E(G)$ with end vertices $v_{i}$ and $v_{j}$ is usually denoted by $v_{i}v_{j}$. For $i=1,2,\ldots,n$, we denote by
$d_{v_{i}}$ (or $d_{i}$ for short) the degree of the vertex $v_{i}$ in $G$, $N(v_{i})$ the set of neighbours
of vertex $v_{i}$ in $G$ and $N[v_{i}]=N(v_{i})\cup\{v_{i}\}$. As usual, let $P_{n}$, $S_{n}$ and $C_{n}$ be the path,
star and cycle of order $n\geq3$.
In molecular graph theory, the topological indices of molecular graphs are used to reflect chemical properties of chemical molecules. There are many topological indices and among them there is a family of vertex-degree-based indices.
The vertex-degree-based index $TI_{f}(G)$ of $G$ with positive symmetric function $f(x,y)$ is defined as
\begin{center}
	$ \displaystyle TI_{f}(G)=\sum\limits_{v_{i}v_{j}\in E(G)}f(d_{i}, d_{j}),$
\end{center}
such as Randi\'{c} index \cite{1}, Atom-Bond Connectivity index \cite{8} and Arithmetic-Geometric index \cite{15} and so on.

Gutman \cite{U4} introduced the Sombor index which is a novel vertex-degree-based topological index in chemical graph
theory. It is defined as
\begin{center}
	$\displaystyle SO(G)=\sum\limits_{v_{i}v_{j}\in E(G)}\sqrt{d_{i}^{2}+d_{j}^{2}}.$
\end{center}
Since the form of the Sombor index corresponds to a $2$-norm, R\'{e}ti et al. \cite{U5} defined a
more general form, i.e., $p$-Sombor index ($p\neq 0$), which was defined as
\begin{center}
	$\displaystyle SO_{p}(G)=\sum\limits_{v_{i}v_{j}\in E(G)}(d_{i}^{p}+d_{j}^{p})^{\frac{1}{p}}.$
\end{center}

Each index maps a molecular graph into a single number. Li \cite{U7} proposed that if we use a matrix
to represent the structure of a molecular graph with weights separately on its pairs of adjacent vertices, it will
completely keep the structural information of the graph. For example, the Randi\'{c} matrix \cite{555,666}, Atom-Bond
Connectivity matrix \cite{10} and Arithmetic-Geometric matrix \cite{16} were considered separately. Based on these examples, Das et al. \cite{U8} proposed the weighted adjacency matrix $A_{f}(G)$, and it is defined as

$$A_{f}(G)(i,j)=\left\{
\begin{aligned}
	&f(d_{i}, d_{j}), \ \ \ \  v_{i}v_{j}\in E(G);\\
	&0,\ \ \ \ \ \ \ \ \ \ \ \ \ \textrm{otherwise}.
\end{aligned}
\right.
$$

Since $G$ is a connected graph, the weighted adjacency matrix $A_{f}(G)$
is an $n\times n$ nonnegative and irreducible symmetric matrix. Thus $\rho(A_{f}(G))$ is exactly the largest eigenvalue of $A_{f}(G)$ and it has a positive eigenvector $\textbf{x}=(x_{1}, x_{2},\ldots, x_{n})^{\intercal}$. As a special case, the adjacency matrix $A(G)$ is $f(x,y)=1$ and the adjacency spectral radius of $G$ is denoted as $\rho(G)$.
If $\displaystyle f'_{x}(x, y)\geq0$ and $\displaystyle f''_{x}(x, y)\geq0$, then $\displaystyle f(x, y)$ is said to be increasing and convex in variable $x$. If $f(x, y)>0$ is increasing and convex in variable $x$ and for any $x_{1}+y_{1}=x_{2}+y_{2}$ and $\mid x_{1}-y_{1}\mid>\mid x_{2}-y_{2}\mid$, $f(x_{1},y_{1})\geq f(x_{2},y_{2})$, then $A_{f}(G)$ is called the weighted adjacency matrix with property $P^{\ast}$ of $G$ in \cite{107}.

Recently, corresponding to the $p$-Sombor index, Liu et al. \cite{U6} defined the $p$-Sombor matrix as $\displaystyle
\textbf{S}_{\textbf{p}}(G)=[s^{p}_{ij}]_{n\times n}$ $(p\neq0)$, where

$$s^{p}_{ij}=\left\{
\begin{aligned}
	&(d_{i}^{p}+d_{j}^{p})^{\frac{1}{p}}, \ \ \ \  v_{i}v_{j}\in E(G);\\
	&0,\ \ \ \ \ \ \ \ \ \ \ \ \ \textrm{otherwise}.
\end{aligned}
\right.
$$
For $p=2$, it is the Sombor matrix. Clearly, if $p\geq1$, then the $p$-Sombor matrix is a weighted adjacency matrix with property $P^{\ast}$. Throughout this paper, we denote $\rho(\textbf{S}_{\textbf{p}}(G))$ the largest eigenvalue of $\textbf{S}_{\textbf{p}}(G)$. Choose the positive eigenvector $\textbf{x}$ such that $\|\textbf{x}\|_{2}=1$ and $x_{i}$ corresponds to the vertex $v_{i}$ and we call the unique unit positive vector $\textbf{x}$ principal eigenvector of $G$.

Gutman \cite{SP1} studied many spectral properties of
the Sombor matrix. Gutman and Gowtham \cite{1181} obtained some results on the coefficients of the characteristic
polynomial and bounds for the energy of Sombor matrix. There is a wealthy literature about Sombor matrix and its related topics, see \cite{12121,12122,12123,12124,12125}. Moreover, Liu et al. \cite{U6} considered some bounds of $p$-Sombor spectral radius and $p$-Sombor spectral spread and the
Nordhaus-Gaddum-type results for $p$-Sombor spectral radius. They obtained that for a tree $T$ of order $n$ and
$p\geq1$,
\begin{center}
	$\displaystyle \rho(\textbf{S}_{\textbf{p}}(T))\leq\rho(\textbf{S}_{\textbf{p}}(S_{n}))$
\end{center}
with equality if and only if $T\cong S_{n}$. Meanwhile, they asked the following problem.
\begin{problem}\label{problem 1.1}\cite{U6}
	What are the structures of the first three maximum and minimum trees, unicyclic and bicyclic graphs for the Sombor spectral radius?
\end{problem}
In this paper, we consider the extremal trees, unicyclic and bicyclic graphs with respect to the $p$-Sombor spectral radii. We characterize completely the extremal graphs with the first three maximum Sombor spectral radii, which answers partially Problem \ref{problem 1.1}. Our main results can be summarized as follows, where the undefined extremal graphs will be characterized in corresponding sections.

\begin{itemize}
	\item[1.] Among all trees of order $n\geq6$, we obtain $S_{n}$, $S_{2,n-2}$, $S_{3,n-3}$ are, respectively, the unique trees with the first three maximum spectral radii for the weighted adjacency matrices with property $P^{\ast}$. In addition, we get
	$\rho(A_{f}(P_{4}))=\rho(A_{f}(S_{2,2}))<\rho(A_{f}(S_{4}))$ and $\rho(A_{f}(P_{5}))<\rho(A_{f}(S_{2,3}))<\rho(A_{f}(S_{5}))$. The result also holds for the $p$-Sombor matrix with $p\geq1$;

	\item[2.] Among all unicyclic graphs of order $n\geq7$, $S_{n}+e$, $U_{1}$, $U_{2}$, $U_{3}$ and $U_{4}$ are, respectively, the unique unicyclic graphs with the first five maximum $p$-Sombor spectral radii for $p\geq2$. For $n=6$, we
	get $S_{6}+e$, $U_{1}$ and $U_{2}$ are, respectively, the unique unicyclic
	graphs with the first three maximum $p$-Sombor spectral radii and the unicyclic
	graph with the fourthly largest $p$-Sombor spectral radius depends on the choice of parameter $p$.
	Furthermore, we obtain
	$\rho(\textbf{S}_{\textbf{p}}(C_{5}))<\rho(\textbf{S}_{\textbf{p}}(U_{3}))<\rho(\textbf{S}_{\textbf{p}}(U_{2}))<\rho(\textbf{S}_{\textbf{p}}(U_{1}))<\rho(\textbf{S}_{\textbf{p}}(S_{5}+e))$
	\noindent for $n=5$;
	\item[3.] Among all bicyclic graphs of order $n\geq6$, we obtain $B_{1}$ is the unique bicyclic graph with the largest spectral radius of the weighted adjacency matrix with property $P^{\ast}$. In addition, $B_{1}$, $B_{2}$ and $B_{3}$ are, respectively, the unique bicyclic graphs with the first three maximum $p$-Sombor spectral radii for $p\geq2$. For $n=5$, we have $B_{1}\cong B_{3}$ is the unique bicyclic graph with the largest $p$-Sombor spectral radius, the bicyclic graphs with the secondly and thirdly largest $p$-Sombor spectral radii depend on the choice of parameter $p$ and they are $B_{2}$ and $B_{4}$.
\end{itemize}

In addition, we give an algorithm to calculate the extremal graphs with the first three minimum Sombor spectral radii among all graphs with order $n$ and size $m$ in Appendix 1.

\section{Preliminary results}
\noindent

In this section, we provide the knowledge of matrix theory on nonnegative matrices, some results on spectra of graphs and a new lemma that will be used in the subsequent sections.

\subsection{Old results}
\noindent

We first recall several well-known results in matrix theory.

\begin{theorem}\label{theorem 2.1}\cite{19}
	Let $A$ and $B$ be both $n\times n$ nonnegative symmetric matrices. Then $\displaystyle \rho(A+B)\geq\rho(A)$. Furthermore, if $A$ is irreducible and $B\neq0$, then $\displaystyle \rho(A+B)>\rho(A)$.
\end{theorem}


\begin{theorem}\label{theorem 2.3}\cite{U1}
	Let $A$ be an $n\times n$ nonnegative and symmetric matrix. Then $\rho(A)\geq\textbf{x}^{\top}A\textbf{x}$ for any unit vector $\textbf{x}$, with equality holds if and only if $\displaystyle A\textbf{x}=\rho(A)\textbf{x}$.
\end{theorem}

\noindent $\textbf{Definition 2.1.}$  Let $A$ be an $n\times n$ real matrix whose rows and columns are indexed by
$X=\{1,2,...,n\}$. We partition $X$ into $\{X_{1},X_{2},...,X_{k}\}$ in order and rewrite $A$ according to
$\{X_{1},X_{2},...,X_{k}\}$ as follows:

\begin{equation*}
	A=\begin{pmatrix}
		A_{1,1} & \cdots & A_{1,k}\\
		\vdots & \ddots & \vdots \\
		A_{k,1} & \cdots & A_{k,k}\\
	\end{pmatrix},
\end{equation*}

\noindent where $A_{i,j}$ is the block of $A$ formed by rows in $X_{i}$ and the columns in $X_{j}$. Let $b_{i,j}$ denote
the average row sum of $A_{i,j}$. Then the matrix $B=[b_{i,j}]$ will be called the \textbf{quotient matrix} of the
partition of $A$. In particular, the partition is called an \textbf{equitable partition} when the row sum of each block
$A_{i,j}$ is constant.

A result on adjacency spectral radius of graphs.

\begin{theorem}\label{theorem 2.4}\cite{U1}
	Let $A\geq 0$ be an irreducible matrix, $B$ be the quotient matrix of an equitable partition of $A$. Then $\displaystyle
	\rho(A)=\rho(B)$.
\end{theorem}



\begin{theorem}\label{theorem 2.7}\cite{U2}
	Let $U(n,n_{1},n_{2})$ be the unicyclic graph of order $n$ as shown in Figure \ref{fig1056}. If $\displaystyle max\{n_{1},n_{2}\}+3\geq \frac{(1+\sqrt{6n+10})^{2}}{9}$, then
	$\displaystyle \rho(U(n,n_{1},n_{2}))\leq \sqrt{max\{n_{1},n_{2}\}+3}$.
\end{theorem}

\begin{figure}[H]
	\centering
	\includegraphics[width=4cm]{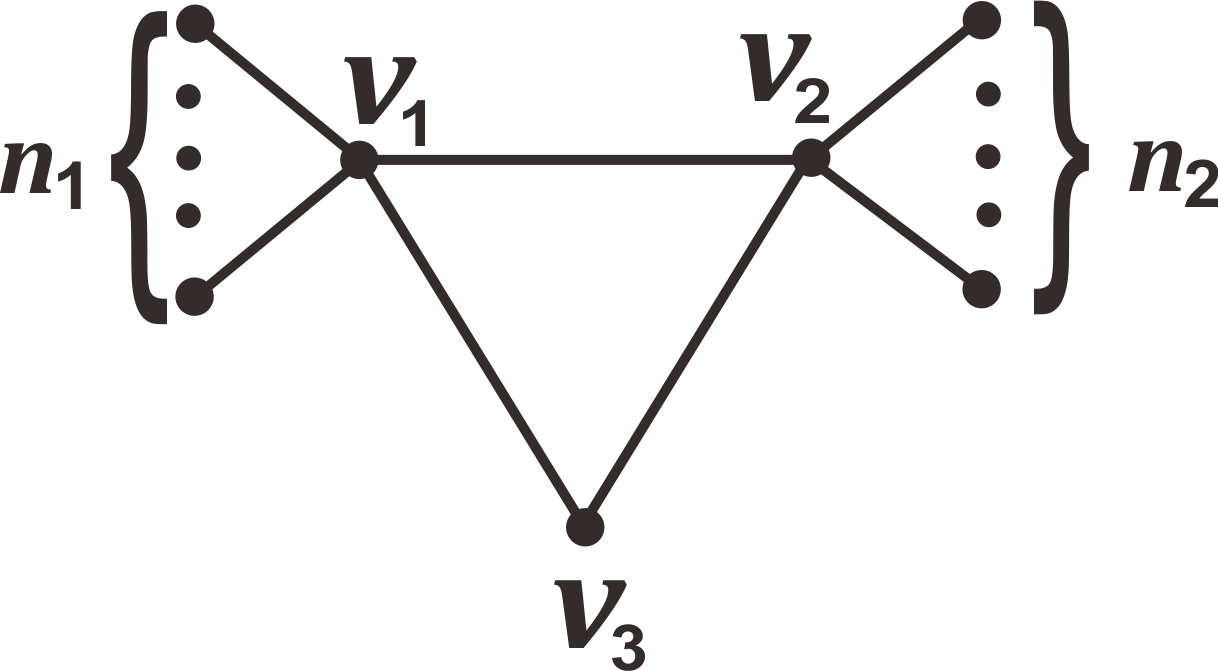}
	\caption{\small  The unicyclic graph $U(n,n_{1},n_{2})$.}
	\label{fig1056}
\end{figure}

\subsection{New Results}

In this subsection, we will list some results about the relation between edge moving and spectral radius of the weighted adjacency matrix with property $P^{\ast}$. These operations will be used repeatedly in the following proofs.

Kelmans \cite{U3} introduced a simple local operation of a graph to describe the relation between edge moving and spectral radius as follows.
\vskip0.2cm
\noindent $\textbf{Definition 2.2.}$ Let $v_{1},v_{2}$ be two vertices of the graph $G$. And we denote: $N_{1}=N(v_{1})-N[v_{2}]$, $N_{2}=N(v_{2})-N[v_{1}]$. We use the Kelmans operation of $G$ as follows: Replace the edge $v_{1}v_{w}$ by a new edge $v_{2}v_{w}$ for all vertices $v_{w}\in N_{1}$ (as shown in Figure \ref{fig2}). In general, we will denote the obtained graph by $G'$.
\begin{figure}[H]
	\centering
	\includegraphics[width=7cm]{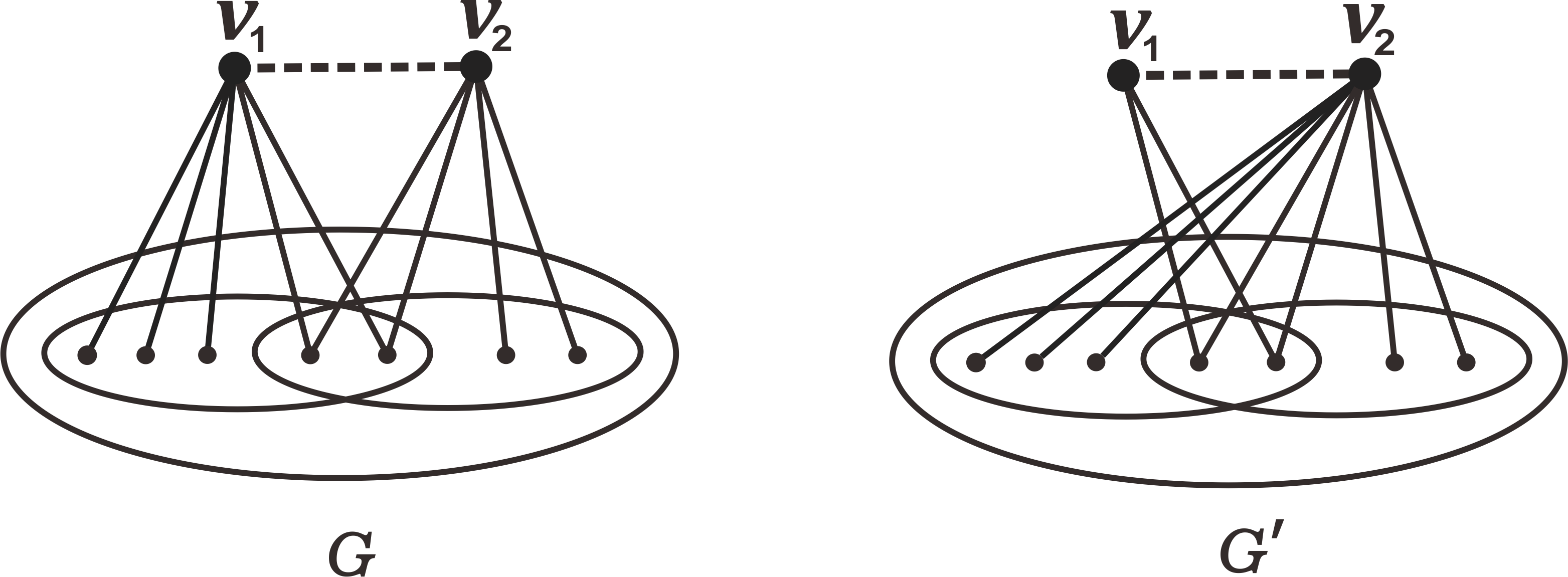}
	\caption{\small  The Kelmans operation.}
	\label{fig2}
\end{figure}

\begin{theorem}\label{theorem 2.8}\cite{107}
	Let $G$ be a connected graph and $G'$ be the graph after a Kelmans operation on any two vertices $v_{1}$ and $v_{2}$ of $G$ as shown in Figure \ref{fig2}. If $G\ncong G'$, then $\displaystyle \rho(A_{f}(G))<\rho(A_{f}(G'))$.
\end{theorem}

The following theorem is a consequence of Theorem \ref{theorem 2.8}.

\begin{theorem}\label{theorem 2.9}\cite{107}
	Let $G$ be a connected graph which consists of a proper induced subgraph $H$ and a tree $T$ of order $m+1$ $(m\geq 2)$ such that $H$ and $T$ has a unique common vertex $u$ and $T$ is not a star with center $u$. Let $G^{\ast}=H+uv_{1}+\cdots+uv_{m}$, where $v_{1},v_{2},\ldots,v_{m}\in V(T)$ are distinct pendent vertices of $G^{\ast}$. Then $\displaystyle \rho(A_{f}(G))<\rho(A_{f}(G^{\ast}))$.
\end{theorem}

For another simple local operation which is described in the following, the Theorem \ref{theorem 2.10} is obtained in \cite{107}.

The graph $F$ is shown in Figure \ref{fig115}. $v_{1},v_{2}$ are adjacent vertices of $F$ and the vertices belong to $N(v_{1})-N[v_{2}]$ (resp. $N(v_{2})-N[v_{1}]$) are all pendent vertices and denoted as $v_{w}$ (resp. $v_{y}$). And $v_{z_{1}},\ldots,v_{z_{m}}$ ($m\geq0$) are the common neighbours of $v_{1}$ and $v_{2}$. In addition, we denote $|N(v_{1})-N[v_{2}]|=n_{1}$ and $|N(v_{2})-N[v_{1}]|=n_{2}$. Without loss of generality, we assume that $1\leq n_{1}\leq n_{2}$. We replace edge $v_{1}v_{w}$ in $F$ by a new edge $v_{2}v_{w}$ for one
pendent vertex $v_{w}\in N(v_{1})-N[v_{2}]$ to obtain a new graph $F''$ as shown in Figure \ref{fig115}. Then we have the following result.
\begin{figure}[H]
	\centering
	\includegraphics[width=9cm]{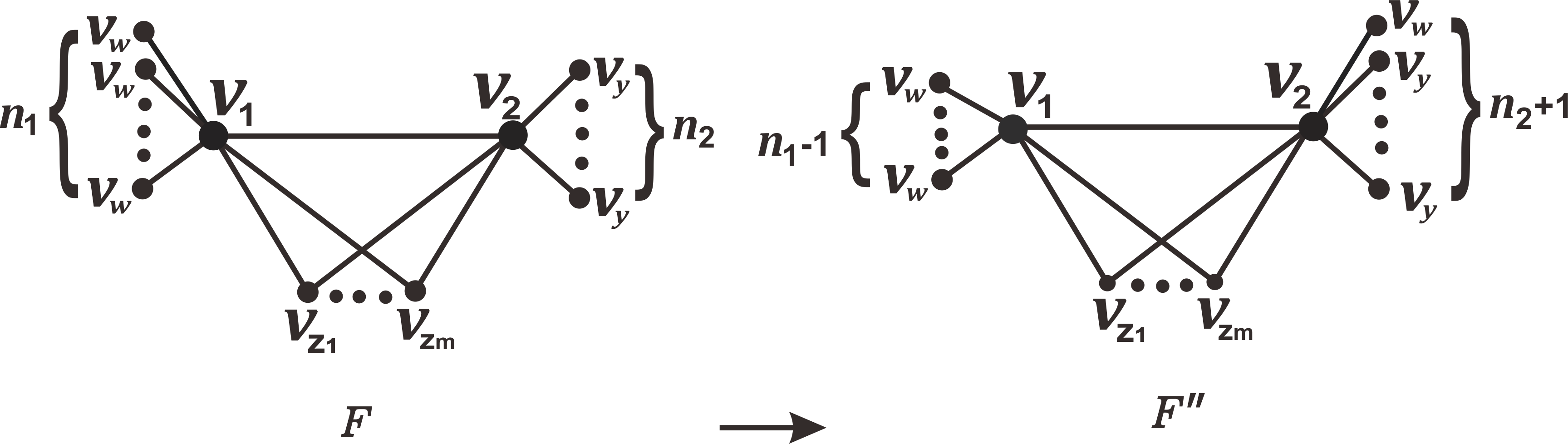}
	\caption{\small  The graphs $F$ and $F''$.}
	\label{fig115}
\end{figure}
\begin{theorem}\label{theorem 2.10}\cite{107}
	Let $G$ be a connected graph of order $n$ which consists of $F$ and a proper induced subgraph $H$ such that $F$ and $H$ have common vertices $v_{z_{1}},\ldots,v_{z_{m}} (0\leq m\leq n-4)$. Replace $F$ by $F''$ to obtain a new connected graph $G''$ (as shown in Figure \ref{fig5}).
	Then $\displaystyle \rho(A_{f}(G))<\rho(A_{f}(G''))$.
\end{theorem}

\begin{figure}[H]
	\centering
	\includegraphics[width=9cm]{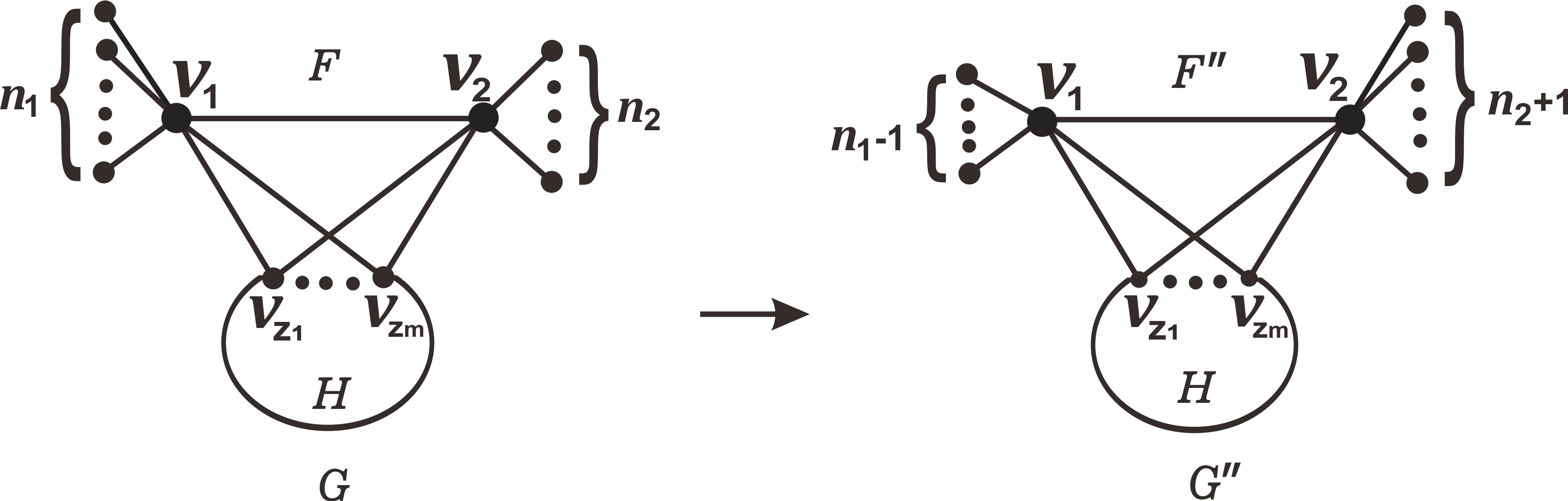}
	\caption{\small  The graphs $G$ and $G''$.}
	\label{fig5}
\end{figure}

\subsection{A lemma}

In this subsection, we give a lemma which will be used to consider the extremal unicyclic and bicyclic graphs in the subsequent sections.

\begin{lemma}\label{lemma 2.1}
	$f(x,y,p)=(x^{p}+y^{p})^{\frac{1}{p}}$ is decreasing in variable $p\geq2$ for $x,y\geq1$.
\end{lemma}
\begin{proof}
	As
	\begin{center}
		$\displaystyle f'_{p}(x,y,p)=\big[\frac{x^{p}\ln x+y^{p}\ln y}{p(x^{p}+y^{p})}-\frac{\ln
			(x^{p}+y^{p})}{p^{2}}\big]\cdot(x^{p}+y^{p})^{\frac{1}{p}}$
	\end{center}
	\noindent  and
	\begin{center}
		$\displaystyle  p\big[x^{p}\ln x+y^{p}\ln y\big]=x^{p}\ln x^{p}+y^{p}\ln y^{p}<(x^{p}+y^{p})\ln (x^{p}+y^{p})$
	\end{center}
	\noindent for $x,y\geq1$ and $p\geq2$, we have $\displaystyle  f'_{p}(x,y,p)<0$. Hence $f(x,y,p)$ is decreasing in
	variable $p\geq2$ for $x,y\geq1$.
\end{proof}

\section{Extremal Trees}

In this section, we will think over the extremal trees with respect to the spectral radius of weighted adjacency matrices with property $P^{\ast}$.

Let $S_{d,n-d}$ be the double star of order $n\geq4$ with two centers $v_{1}$, $v_{2}$
such that $d_{1}=d$ and $d_{2}=n-d$ where $\displaystyle 2\leq d\leq \lfloor\frac{n}{2}\rfloor$ as shown in Figure \ref{fig1}. Then we have the following theorem.

\begin{figure}[H]
	\centering
	\includegraphics[width=5cm]{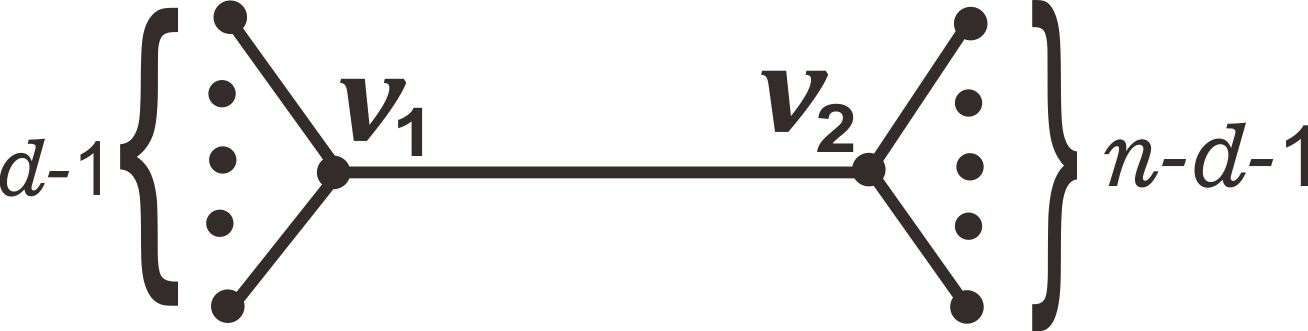}
	\caption{\small  The double star $S_{d,n-d}$ of order $n\geq6$.}
	\label{fig1}
\end{figure}

\begin{theorem}\label{theorem 4.1}
	Among all trees of order $n\geq6$, $S_{n}$, $S_{2,n-2}$, $S_{3,n-3}$ are, respectively, the unique trees with the first three maximum spectral radii of the weighted adjacency matrices with property $P^{\ast}$.
\end{theorem}
\begin{proof} We consider the following two cases to prove the theorem.
	
	$\textbf{Case 1.}$ $T$ is a double star and $T\ncong S_{n}, S_{2,n-2}, S_{3,n-3}$.
	
	According to Theorem \ref{theorem 2.10}, we obtain
	\begin{center}
		$\rho(A_{f}(T))<\rho(A_{f}(S_{3,n-3}))<\rho(A_{f}(S_{2,n-2}))<\rho(A_{f}(S_{n}))$.
	\end{center}

	$\textbf{Case 2.}$ $T$ is not a double star.
	
	Since $T$ is not a double star, $T$ contains $P_{5}$ as its subgraph as shown in Figure \ref{fig12112}. Without loss of generality, we assume that $d_{2}\geq d_{4}\geq2$.
	\begin{figure}[H]
		\centering
		\includegraphics[width=5cm]{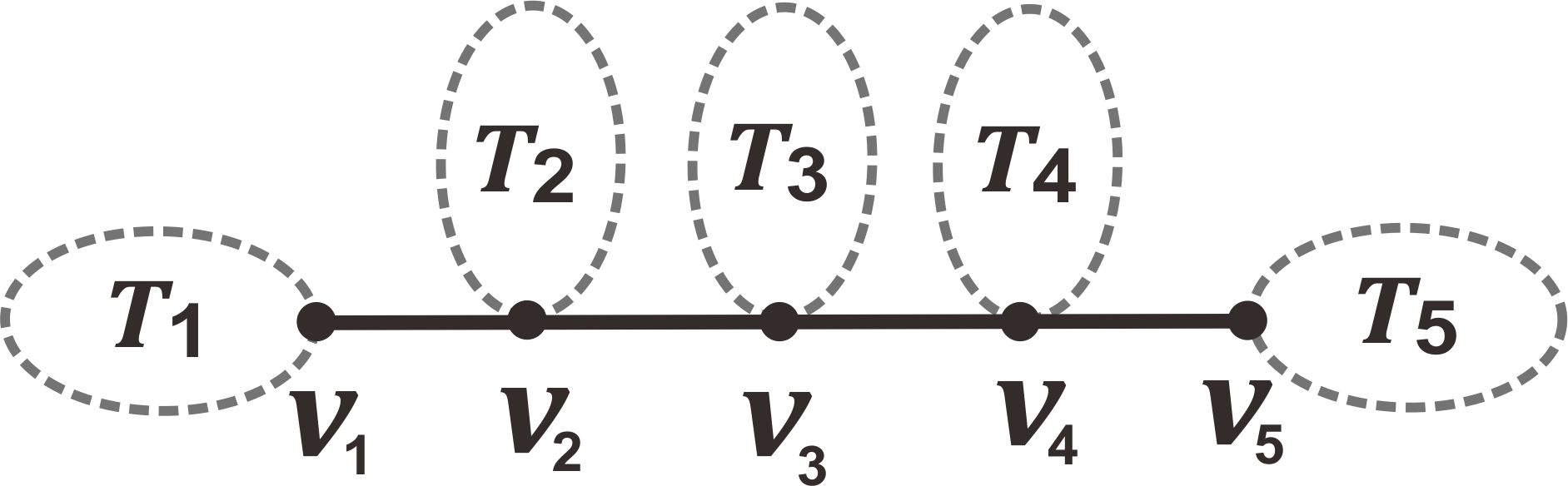}
		\caption{\small  The tree $T$ for Subcase 2.2.}
		\label{fig12112}
	\end{figure}
	$\textbf{Subcase 2.1.}$ $d_{2}\geq3$.
	
	We use Kelmans operation on the adjacent vertices $v_{3}$ and $v_{4}$ as shown in Figure \ref{fig161}. By Theorems \ref{theorem 2.8}, \ref{theorem 2.9} and \ref{theorem 2.10}, we obtain $\rho(A_{f}(T))<\rho(A_{f}(S_{3,n-3}))$.
	
	\begin{figure}[H]
		\centering
		\includegraphics[width=10cm]{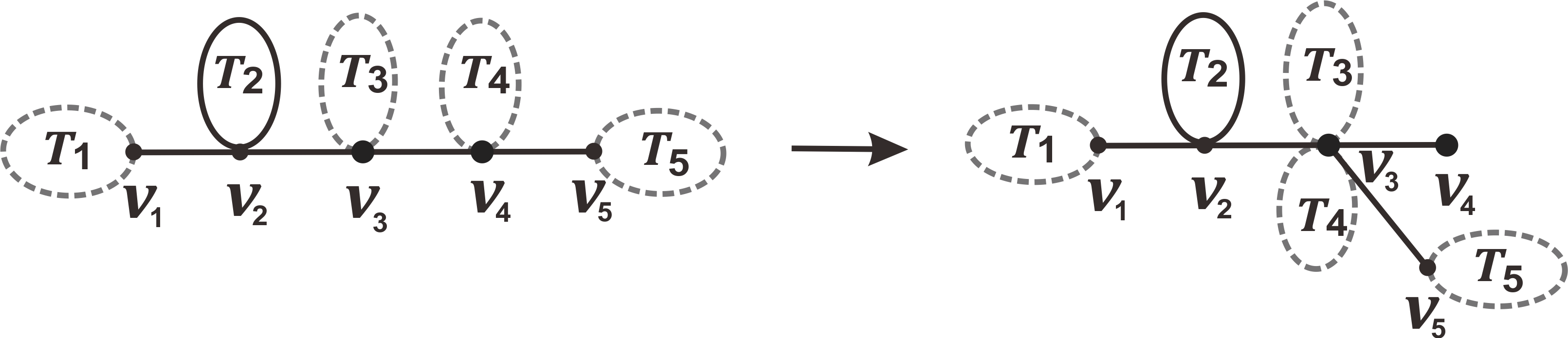}
		\caption{\small  The Kelmans operation for Subcase 2.1.}
		\label{fig161}
	\end{figure}

	$\textbf{Subcase 2.2.}$  $d_{2}=d_{4}=2$.
	
	Without loss of generality, we assume that $d_{1}\geq d_{5}\geq1$.  If $d_{3}=2$, then $d_{1}\geq2$. We use Kelmans operation on the adjacent vertices $v_{3}$ and $v_{4}$ as shown in Figure \ref{fig162}. Analogically, we have $\rho(A_{f}(T))<\rho(A_{f}(S_{3,n-3}))$.
	
	\begin{figure}[H]
		\centering
		\includegraphics[width=10cm]{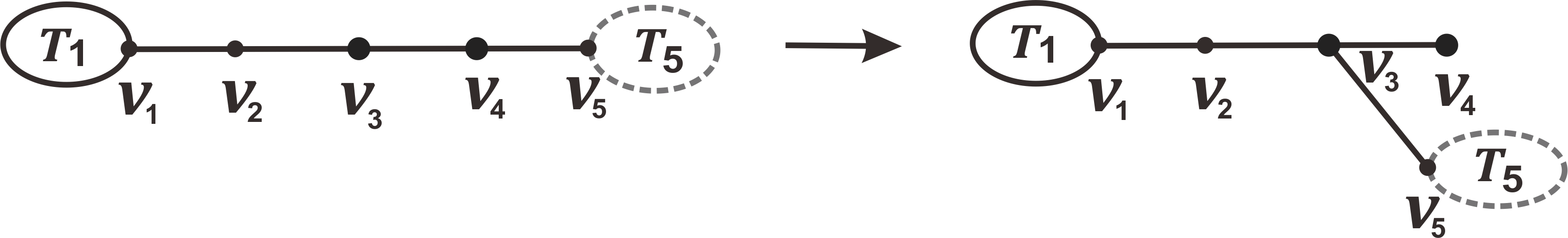}
		\caption{\small  The Kelmans operation on the adjacent vertices $v_{3}$ and $v_{4}$.}
		\label{fig162}
	\end{figure}
	
	Otherwise, $d_{3}>2$. We use Kelmans operation on the vertices $v_{2}$ and $v_{4}$ as shown in Figure \ref{fig163}. We also can obtain $\rho(A_{f}(T))<\rho(A_{f}(S_{3,n-3}))$.
	\begin{figure}[H]
		\centering
		\includegraphics[width=10cm]{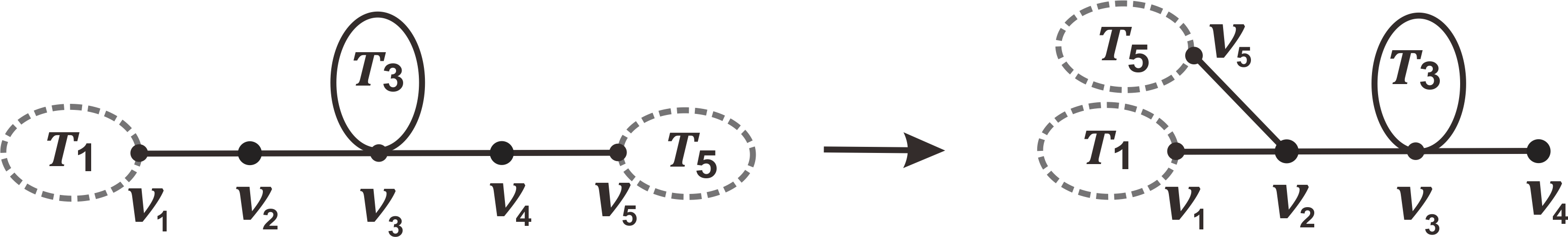}
		\caption{\small  The Kelmans operation on the vertices $v_{2}$ and $v_{4}$.}
		\label{fig163}
	\end{figure}
	Hence, we have $S_{n}$, $S_{2,n-2}$, $S_{3,n-3}$
	are, respectively, the unique trees with the first three maximum spectral radii of the weighted adjacency matrices with property $P^{\ast}$.\textbf{}\end{proof}

The result also holds for the $p$-Sombor matrix with $p\geq1$. Thus we have the following theorem.

\begin{theorem}\label{theorem 4.2}
	Among all trees of order $n\geq6$, $S_{n}$, $S_{2,n-2}$, $S_{3,n-3}$ are, respectively, the unique trees with the first three maximum $p$-Sombor spectral radii for $p\geq1$.
\end{theorem}

\noindent  \textbf{Remark 3.1.} For the weighted adjacency matrix with property $P^{\ast}$ and $n=4,5$, we have
\begin{center}
	$\rho(A_{f}(P_{4}))=\rho(A_{f}(S_{2,2}))<\rho(A_{f}(S_{4}))$
\end{center}
\noindent and
\begin{center}
	$\rho(A_{f}(P_{5}))<\rho(A_{f}(S_{2,3}))<\rho(A_{f}(S_{5}))$.
\end{center}
It is a direct consequence of Theorem \ref{theorem 2.8} and also holds for the $p$-Sombor matrix with $p\geq1$.

\noindent  \textbf{Remark 3.2.} Liu et al. \cite{U6} has obtained $S_{n}$ is the unique tree with the largest $p$-Sombor spectral radius for $p\geq1$. We consider the trees with the first three maximum $p$-Sombor spectral radii with $p\geq1$. For the largest case, we obtain the same result.

\section{Extremal unicyclic graphs}

From now on, we assume that $\displaystyle f(x,y,p)=(x^{p}+y^{p})^{\frac{1}{p}}$ with $x,y\geq1$ and $p\geq2$, then $\displaystyle f(x,y,p)>0$ is strictly increasing and convex in variable $x$. And we consider the $p$-Sombor matrix $\displaystyle
\textbf{S}_{\textbf{p}}(G)=[f(d_{i},d_{j},p)]_{n\times n}$ with $p\geq2$ which is the weighted adjacency matrix with property $P^{\ast}$ of $G$.

In this section, we will think about the extremal unicyclic graphs with respect to the $p$-Sombor spectral radii.

The unicyclic graph $S_{n}+e$ with $n\geq4$ is obtained from $S_{n}$ by adding an edge as shown in Figure \ref{fig102} (the first one). The unicyclic graphs $U_{1}$, $U_{2}$, $U_{3}$, $U_{4}$ and $U_{5}$ that will be considered in this section are also shown in Figure \ref{fig102}.

\begin{figure}[H]
	\centering
	\includegraphics[width=10cm]{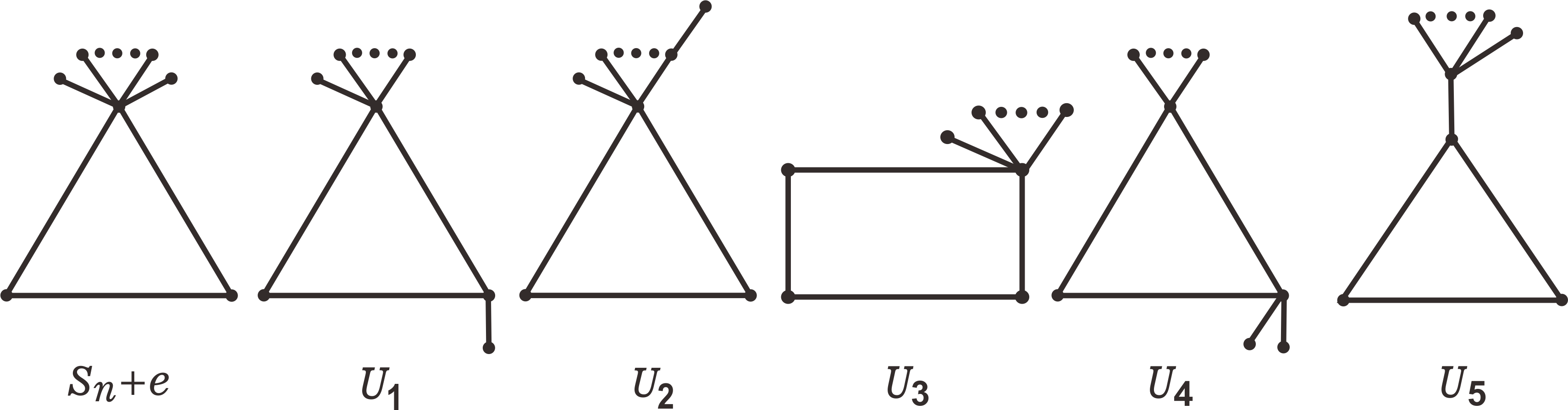}
	\caption{\small  The unicyclic graphs $S_{n}+e$, $U_{1}$, $U_{2}$, $U_{3}$, $U_{4}$ and $U_{5}$ of order $n\geq7$.}
	\label{fig102}
\end{figure}

We first recall a result of the unicyclic graphs with respect to spectral radius of weighted adjacency matrices with property $P^{\ast}$.

\begin{theorem}\label{theorem 2.11}\cite{107} If $G$ is a unicyclic graph of order $n\geq7$ and is not isomorphic to the graphs $S_{n}+e$, $U_{1}$, $U_{2}$, $U_{3}$, $U_{4}$ and $U_{5}$ (as shown in Figure \ref{fig102}), then $\displaystyle \rho(A_{f}(G))<\rho(A_{f}(U_{4}))$. In addition, $\displaystyle \rho(A_{f}(U_{3}))<\rho(A_{f}(U_{2}))<\rho(A_{f}(U_{1}))<\rho(A_{f}(S_{n}+e))$.
\end{theorem}

Next, we consider the unique unicyclic graphs with the first five maximum $p$-Sombor spectral radii.

\begin{lemma}\label{lemma 4.1}
	$U_{4}$ and $U_{5}$ are the unicyclic graphs of order $n\geq7$ as shown in Figure \ref{fig102}. Then we have $\displaystyle \rho(\textbf{S}_{\textbf{p}}(U_{5}))<\rho(\textbf{S}_{\textbf{p}}(U_{4}))$.
\end{lemma}
\begin{proof} We consider the following two cases.
	
	$\textbf{Case 1.}$ $n\geq10$.
	
	Suppose that the vertices of $U_{5}$ are denoted as Figure \ref{fig10} and $\textbf{x}$ be the principal eigenvector of $U_{5}$. We assume that the pendent vertices belonging to $N(v_{1})$ are denoted as $v_{w}$. As $\textbf{S}_{\textbf{p}}(G)\textbf{x}=\rho(\textbf{S}_{\textbf{p}}(G))\textbf{x}$, the entries corresponding to these pendent vertices in $\textbf{x}$ are equal.
	
	\begin{figure}[H]
		\centering
		\includegraphics[width=7cm]{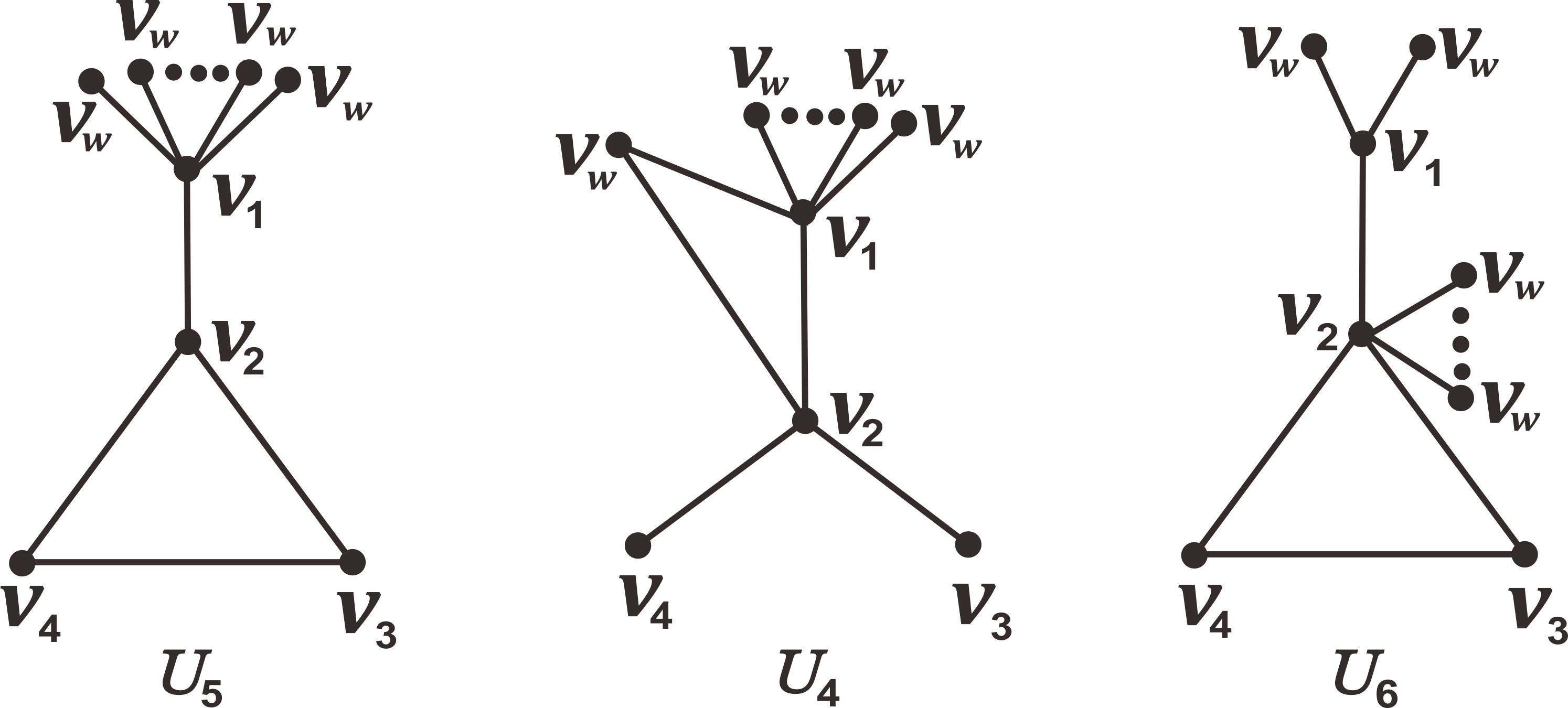}
		\caption{\small  The unicyclic graphs $U_{4}$, $U_{5}$ and $U_{6}$ of order $n$.}
		\label{fig10}
	\end{figure}
	
	Moreover, because $\textbf{S}_{\textbf{p}}(U_{5})\textbf{x}=\rho(\textbf{S}_{\textbf{p}}(U_{5}))\textbf{x}$, we get
	\begin{center}
		\noindent $\displaystyle \rho(\textbf{S}_{\textbf{p}}(U_{5}))x_{1}=(n-4)f(1,n-3,p)x_{w}+f(3,n-3,p)x_{2},$
	\end{center}
	\noindent and
	\begin{center}
		\noindent $\displaystyle \rho(\textbf{S}_{\textbf{p}}(U_{5}))x_{2}=f(3,n-3,p)x_{1}+2f(2,3,p)x_{3}.$
	\end{center}
	Hence if $x_{1}\leq x_{2}$, then $x_{w}<x_{3}$. Meanwhile,
	\begin{center}
		\noindent $\displaystyle \rho(\textbf{S}_{\textbf{p}}(U_{5}))x_{3}=f(2,3,p)x_{2}+f(2,2,p)x_{4}$
	\end{center}
	and
	\begin{center}
		$\displaystyle \rho(\textbf{S}_{\textbf{p}}(U_{5}))x_{w}=f(1,n-3,p)x_{1}.$
	\end{center}
	We obtain if $x_{1}>x_{2}$, then $x_{2}>x_{3}$ for $n\geq10$. Thus we have $x_{1}>x_{2}>x_{3}$. According to Lemma \ref{lemma 2.1}, we obtain $\displaystyle f(1,n-3,p)>f(2,3,p)+f(2,2,p)$ for $n\geq10$. It is easy to see $\displaystyle x_{1}>x_{2}>x_{w}>x_{3}$. Then we consider the following two cases.
	
	\textbf{Subcase 1.1.} $\displaystyle x_{1}>x_{2}$.
	
	Replace edge $v_{3}v_{4}$ in $U_{5}$ by a new edge
	$v_{2}v_{w}$ to obtain $U_{4}$ as shown in Figure \ref{fig10}. Then we have
	$$\begin{aligned}
		&\quad
		\textbf{x}^{\top}\textbf{S}_{\textbf{p}}(U_{4})\textbf{x}-\textbf{x}^{\top}\textbf{S}_{\textbf{p}}(U_{5})\textbf{x}=\\
		& 2\big[f(4,n-3,p)-f(3,n-3,p)\big]x_{1}x_{2}+4\big[f(1,4,p)-f(2,3,p)\big]x_{2}x_{3}+\\
		& 2\big[f(2,n-3,p)-f(1,n-3,p)\big]x_{1}x_{w}+2\big[f(2,4,p)x_{2}x_{w}-f(2,2,p)x_{3}x_{4}\big]>0.
	\end{aligned}$$
	
	\noindent  Thus by Theorem \ref{theorem 2.3}, we obtain $\displaystyle \rho(\textbf{S}_{\textbf{p}}(U_{5}))<\rho(\textbf{S}_{\textbf{p}}(U_{4}))$.
	
	\textbf{Case 1.2.} $\displaystyle x_{1}\leq x_{2}$.
	
	Replace edge $v_{1}v_{w}$ in $U_{5}$ by a new edge $v_{2}v_{w}$ for $n-6$ pendent vertices $v_{w}\in N(v_{1})$ to obtain a new
	graph $U_{6}$ as shown in Figure \ref{fig10}. We get
	$$\begin{aligned}
		&\quad
		\textbf{x}^{\top}\textbf{S}_{\textbf{p}}(U_{6})\textbf{x}-\textbf{x}^{\top}\textbf{S}_{\textbf{p}}(U_{4})\textbf{x}=\\
		& 4\big[f(1,3,p)-f(1,n-3,p)\big]x_{1}x_{w}+4\big[f(2,n-3,p)-f(2,3,p)\big]x_{2}x_{3}\\
		+&2(n-6)\big[f(1,n-3,p)x_{2}x_{w}-f(1,n-3,p)x_{1}x_{w}\big].
	\end{aligned}$$
	
	\noindent  Because $\displaystyle x_{1}\leq x_{2}$, we have
	
	$(n-4)f(1,n-3,p)x_{w}+f(3,n-3,p)x_{2}\leq f(3,n-3,p)x_{1}+2f(2,3,p)x_{3}$,
	
	\noindent it follows that
	
	$(n-4)f(1,n-3,p)x_{w}\leq 2f(2,3,p)x_{3}$, i.e. $\displaystyle x_{w}\leq \frac{2f(2,3,p)}{(n-4)f(1,n-3,p)}x_{3}$.
	
	\noindent Thus for $n\geq10$, we have
	$$\begin{aligned}
		&\big[f(1,3,p)-f(1,n-3,p)\big]x_{1}x_{w}+\big[f(2,n-3,p)-f(2,3,p)\big]x_{2}x_{3}\\
		>&\big\{\big[f(1,3,p)-f(1,n-3,p)\big]\frac{2f(2,3,p)}{(n-4)f(1,n-3,p)}+f(2,n-3,p)-f(2,3,p)\big\}x_{2}x_{3}\\
		=&\frac{1}{(n-4)f(1,n-3,p)}\big\{2f(2,3,p)\big[f(1,3,p)-f(1,n-3,p)\big]\\
		&+(n-4)f(1,n-3,p)\big[f(2,n-3,p)-f(2,3,p)\big]\big\}x_{2}x_{3}\\
		=&\frac{1}{(n-4)f(1,n-3,p)}\big\{2f(2,3,p)f(1,3,p)+f(1,n-3,p)\big[(n-4)f(2,n-3,p)\\
		-&2f(2,3,p)-(n-4)f(2,3,p)\big]\big\}x_{2}x_{3}\\
		>&\frac{2f(2,3,p)f(1,3,p)+f(1,n-3,p)\big[n^{2}-(7+\sqrt{13})n+2\sqrt{13}+12\big]}{(n-4)f(1,n-3,p)}x_{2}x_{3} \ (\textrm{by} \ \textrm{Lemma} \ \ref{lemma 2.1})\\
		>&0.
	\end{aligned}$$
	Thus we get
	$\textbf{x}^{\top}\textbf{S}_{\textbf{p}}(U_{6})\textbf{x}-\textbf{x}^{\top}\textbf{S}_{\textbf{p}}(U_{5})\textbf{x}>0$.
	By Theorem \ref{theorem 2.3}, we have $\displaystyle \rho(\textbf{S}_{\textbf{p}}(U_{5}))<\rho(\textbf{S}_{\textbf{p}}(U_{6}))$.
	According to Theorem \ref{theorem 2.8}, we obtain $\displaystyle \rho(\textbf{S}_{\textbf{p}}(U_{6}))<
	\rho(\textbf{S}_{\textbf{p}}(U_{4}))$. Thus we get $\displaystyle \rho(\textbf{S}_{\textbf{p}}(U_{5}))<
	\rho(\textbf{S}_{\textbf{p}}(U_{4}))$ for $n\geq10$.
	
	$\textbf{Subcase 2.}$ $7\leq n\leq9$.
	
	With the aid of MATLAB, we obtain $\displaystyle
	\rho(\textbf{S}_{\textbf{p=2}}(U_{5}))<\lim_{p \to \infty}\rho(\textbf{S}_{\textbf{p}}(U_{4}))$ for $7\leq n\leq9$. Then by Lemma \ref{lemma 2.1}, we obtain $\displaystyle \rho(\textbf{S}_{\textbf{p}}(U_{5}))<\rho(\textbf{S}_{\textbf{p}}(U_{4}))$ for $7\leq n\leq9$.
	
	This completes the proof.
\end{proof}

\begin{lemma}\label{lemma 4.2}
	$U_{3}$ and $U_{4}$ are the unicyclic graphs of order $n\geq7$ as shown in Figure \ref{fig102}.
	Then we have $\displaystyle
	\rho(\textbf{S}_{\textbf{p}}(U_{4}))<\rho(\textbf{S}_{\textbf{p}}(U_{3}))$.
\end{lemma}
\begin{proof}
	Let the vertices of $U_{3}$ be denoted as Figure \ref{fig23120}.
	\begin{figure}[H]
		\centering
		\includegraphics[width=3cm]{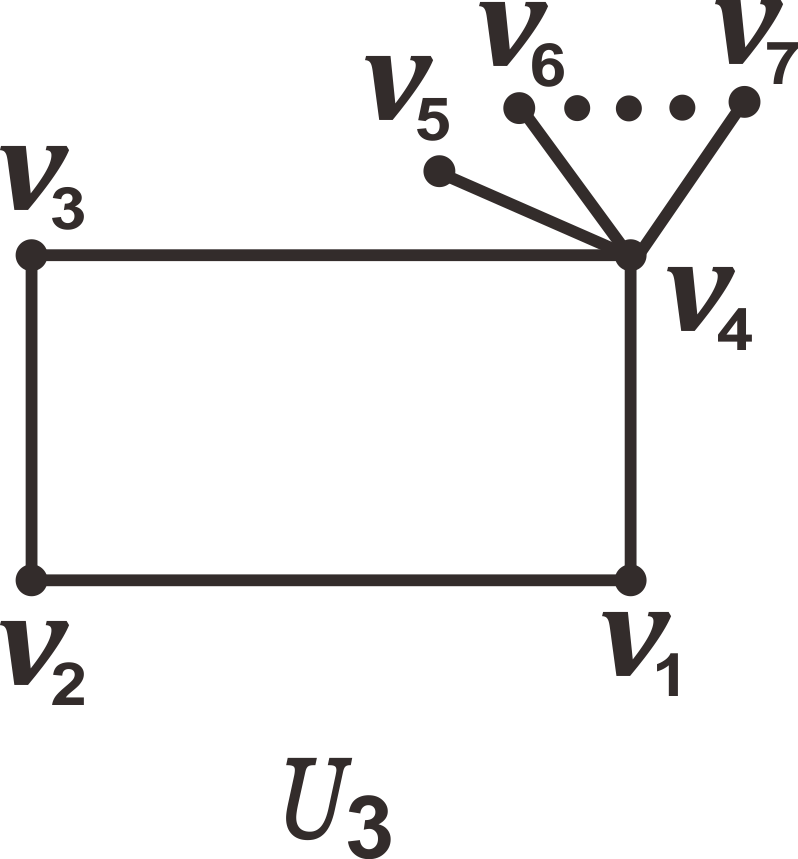}
		\caption{\small  The unicyclic graph $U_{3}$ of order $n$.}
		\label{fig23120}
	\end{figure}
	The quotient matrix of the equitable partition $\{\{v_{1}, v_{3}\},\{v_{2}\},\{v_{4}\},\{v_{5},\ldots,v_{n}\}\}$ of $\textbf{S}_{\textbf{p}}(U_{3})$ is
	\begin{equation*}
		\begin{pmatrix}
			0 & f(2,2) & f(2,n-2) & 0\\
			2f(2,2)& 0 & 0 & 0 \\
			2f(2,n-2)& 0 & 0 & (n-4)f(1,n-2) \\
			0 & 0 & f(1,n-2) & 0\\
		\end{pmatrix}.
	\end{equation*}
	We assume that $\displaystyle P(x, U_{3})$ is the characteristic polynomial of the quotient matrix. By calculation, we obtain
	\begin{align}
		P(x, U_{3})&=x^{4}-\big[2f^{2}(2,2,p)+2f^{2}(2,n-2,p)+(n-4)f^{2}(1,n-2,p)\big]x^{2}\notag\\
		&+2(n-4)f^{2}(2,2,p)f^{2}(1,n-2,p)\notag.
	\end{align}
	\noindent  Since for any $x$, we have
	\begin{eqnarray*}
		P(x, U_{3})&<&x^{4}-\big[2f^{2}(2,2,p)+2f^{2}(2,n-2,p)+(n-4)f^{2}(1,n-2,p)\big]x^{2}\\
		&+&2f^{2}(2,2,p)\big[2f^{2}(2,n-2,p)+(n-4)f^{2}(1,n-2,p)\big],
	\end{eqnarray*}
	\noindent it follows that
	\begin{center}
		$\displaystyle 2f^{2}(2,n-2,p)+(n-4)f^{2}(1,n-2,p)<\rho(\textbf{S}_{\textbf{p}}(U_{3}))^{2}.$
	\end{center}
	
	We consider the following two cases.
	
	$\textbf{Case 1.}$ $n\geq17$.
	
	According to Theorems \ref{theorem 2.1} and \ref{theorem 2.7}, it follows that
	\begin{center}
		$\displaystyle \rho(\textbf{S}_{\textbf{p}}(U_{4}))^{2}<(n-2)f^{2}(4,n-3,p)$.
	\end{center}
	
	\noindent Then we only need to prove
	\begin{center}
		$\displaystyle (n-2)f^{2}(4,n-3,p)<2f^{2}(2,n-2,p)+(n-4)f^{2}(1,n-2,p).$
	\end{center}
	\noindent Let
	\begin{eqnarray*}
		Q(n,p)&=&\displaystyle  2f^{2}(2,n-2,p)+(n-4)f^{2}(1,n-2,p)-(n-2)f^{2}(4,n-3,p)\\
		&=&2\big[2^{p}+(n-2)^{p}\big]^{\frac{2}{p}}+(n-4)\big[1+(n-2)^{p}\big]^{\frac{2}{p}}-(n-2)\big[4^{p}+(n-3)^{p}\big]^{\frac{2}{p}}.
	\end{eqnarray*}
	\noindent  We consider
	\begin{center}
		$\displaystyle h(x,y)=1+(x-2)^{y}-\big[4^{y}+(x-3)^{y}\big]$.
	\end{center}
	\noindent  Since
	\begin{center}
		$\displaystyle h'_{x}(x,y)=y(x-2)^{y-1}-y(x-3)^{y-1}>0$
	\end{center}
	\noindent  for $y\geq2$ and $x\geq17$, we get
	\begin{center}
		$\displaystyle h(x,y)\geq h(17,y)=1+15^{y}-4^{y}-14^{y}$.
	\end{center}
	As
	\begin{eqnarray*}
		h'_{y}(17,y)&=&15^{y}\ln 15-4^{y}\ln 4-14^{y}\ln 14\\
		&\geq&15^{y-2}\big[15^{2}\ln 15-4^{2}\ln 4-14^{2}\ln 14\big]>0
	\end{eqnarray*}
	\noindent for $y\geq 2$, we obtain $\displaystyle h(x,y)\geq h(17,2)>0$. It is easy to show $\displaystyle Q(n,p)>0$.
	
	Consequently, we infer that $\displaystyle
	\rho(\textbf{S}_{\textbf{p}}(U_{4}))<\rho(\textbf{S}_{\textbf{p}}(U_{3}))$ for $n\geq17$.
	
	$\textbf{Case 2.}$ $7\leq n\leq16$.
	
	Similarly, by Lemma \ref{lemma 2.1} and with the aid of MATLAB, we obtain $\displaystyle \rho(\textbf{S}_{\textbf{p}}(U_{4}))<\rho(\textbf{S}_{\textbf{p}}(U_{3}))$. The proof of this result is quite similar to that given for Case 2 of Lemma \ref{lemma 4.1}. It is not difficult but is too long to give here.
	
	The proof of the lemma is now complete.
\end{proof}

\begin{theorem}\label{theorem 3.2}
	Among all unicyclic graphs of order $n\geq7$, $S_{n}+e$, $U_{1}$, $U_{2}$, $U_{3}$ and $U_{4}$ (see in Figure
	\ref{fig102}) are, respectively, the unique unicyclic graphs with the first five maximum $p$-Sombor spectral radii.
\end{theorem}
\begin{proof} According to Theorem \ref{theorem 2.11} and Lemmas \ref{lemma 4.1} and \ref{lemma 4.2}, we
	get $S_{n}+e$, $U_{1}$, $U_{2}$, $U_{3}$ and $U_{4}$ as shown in Figure \ref{fig102} are, respectively, the unique unicyclic
	graphs with the first five maximum $p$-Sombor spectral radii.
\end{proof}

\noindent  \textbf{Remark 4.1.}
For $n=5$, we have $S_{5}+e\cong U_{4}$ and $U_{2}\cong U_{5}$.
Under the help of MATLAB and by Theorem \ref{theorem 2.8}, we obtain
\begin{center}
	$\rho(\textbf{S}_{\textbf{p}}(S_{5}+e))>\rho(\textbf{S}_{\textbf{p}}(U_{1}))>\rho(\textbf{S}_{\textbf{p}}(U_{2}))>\rho(\textbf{S}_{\textbf{p}}(U_{3}))>\rho(\textbf{S}_{\textbf{p}}(C_{5}))$.
\end{center}
And for $n=6$, we
get $S_{6}+e$, $U_{1}$ and $U_{2}$ are, respectively, the unique unicyclic
graphs with the first three maximum $p$-Sombor spectral radii. Moreover, we find the unicyclic
graph with the fourthly largest $p$-Sombor spectral radius depends on the choice of parameter $p$.

\section{Extremal bicyclic graphs}

In this section, we will meditate the extremal bicyclic graphs with respect to the spectral radius of weighted adjacency matrices with property $P^{\ast}$ and the $p$-Sombor spectral radius.

The bicyclic graph obtained from two cycles $C_{q}$ and $C_{t}$ ($q\geq t\geq3$) with a common
path $P_{l}$ such that $2\leq l\leq t-l+2$ is denoted by $P(q,l,t)$ as shown in Figure \ref{fig1051} (left). Set $\mathcal{P}=\{P(q,l,t)\mid q\geq t\geq3, 2\leq l\leq t-l+2\}$.

In addition, suppose that $C_{q}$ and $C_{t}$ are two vertex-disjoint cycles ($q\geq t\geq3$), $z_{1}$ is a vertex of $C_{q}$ and $z_{l}$ is a vertex of $C_{t}$. We join $z_{1}$ and $z_{l}$ by a path $z_{1}z_{2}\ldots z_{l}$ of length $l-1$, where $l\geq1$ and $l=1$ means identifying $z_{1}$ with $z_{l}$. The resulting graph is denoted by $B(q,l,t)$ as shown in Figure \ref{fig1051} (right). Set $\mathcal{B}=\{B(q,l,t)\mid q\geq t\geq3, l\geq1\}$.
\begin{figure}[H]
	\centering
	\includegraphics[width=10cm]{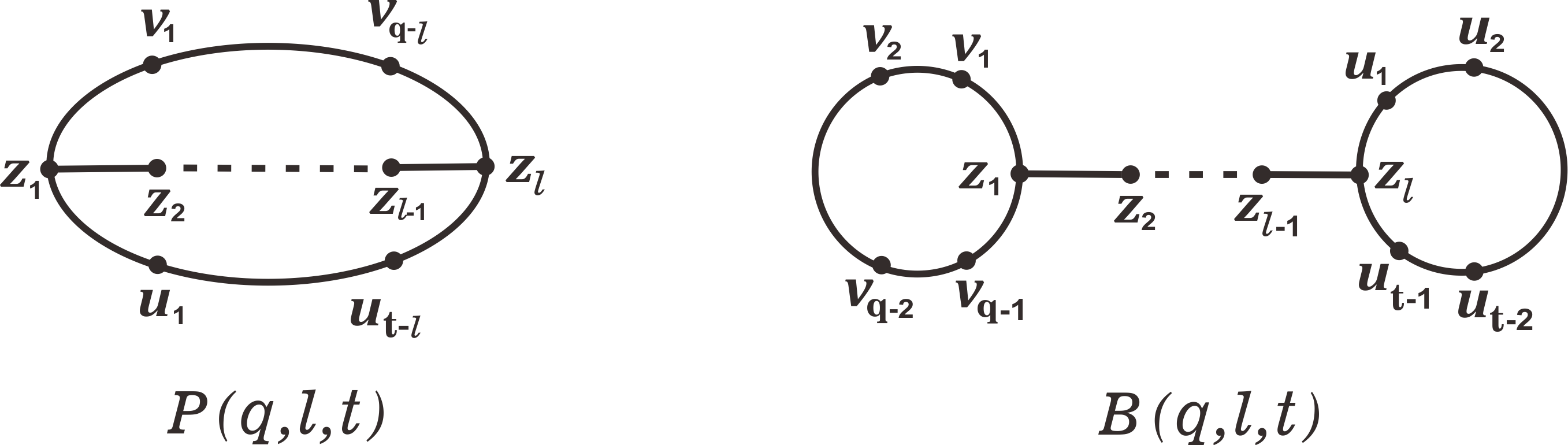}
	\caption{\small  The bicyclic graphs $P(q,l,t)$ and $B(q,l,t)$.}
	\label{fig1051}
\end{figure}

For a bicyclic graph $G$, the base of $G$, denoted by $\hat{G}$, is the unique minimal bicyclic subgraph of $G$.
Suppose that $\mathcal{B}(n)$ is the set of all bicyclic graphs of order $n$. We can define the following two classes of bicyclic graphs with order $n$:
\begin{center}
	$\mathcal{B}_{1}(n)=\{G\in\mathcal{B}(n)\mid \hat{G}\in\mathcal{P}\}$ and $\mathcal{B}_{2}(n)=\{G\in\mathcal{B}(n)\mid \hat{G}\in\mathcal{B}\}$.
\end{center}
\noindent  It is easy to see that $\mathcal{B}(n)=\mathcal{B}_{1}(n)\bigcup\mathcal{B}_{2}(n)$ and $\mathcal{B}_{1}(n)\cap\mathcal{B}_{2}(n)=\emptyset$. In addition, the bicyclic graphs $B_{1}$, $B_{2}$, $B_{3}$ and $B_{4}$ that will be considered in this section are shown in Figure \ref{fig3}. And it is obviously that $B_{1}$, $B_{3}$, $B_{4}\in\mathcal{B}_{1}(n)$ and $B_{2}\in\mathcal{B}_{2}(n)$.

\begin{figure}[H]
	\centering
	\includegraphics[width=8cm]{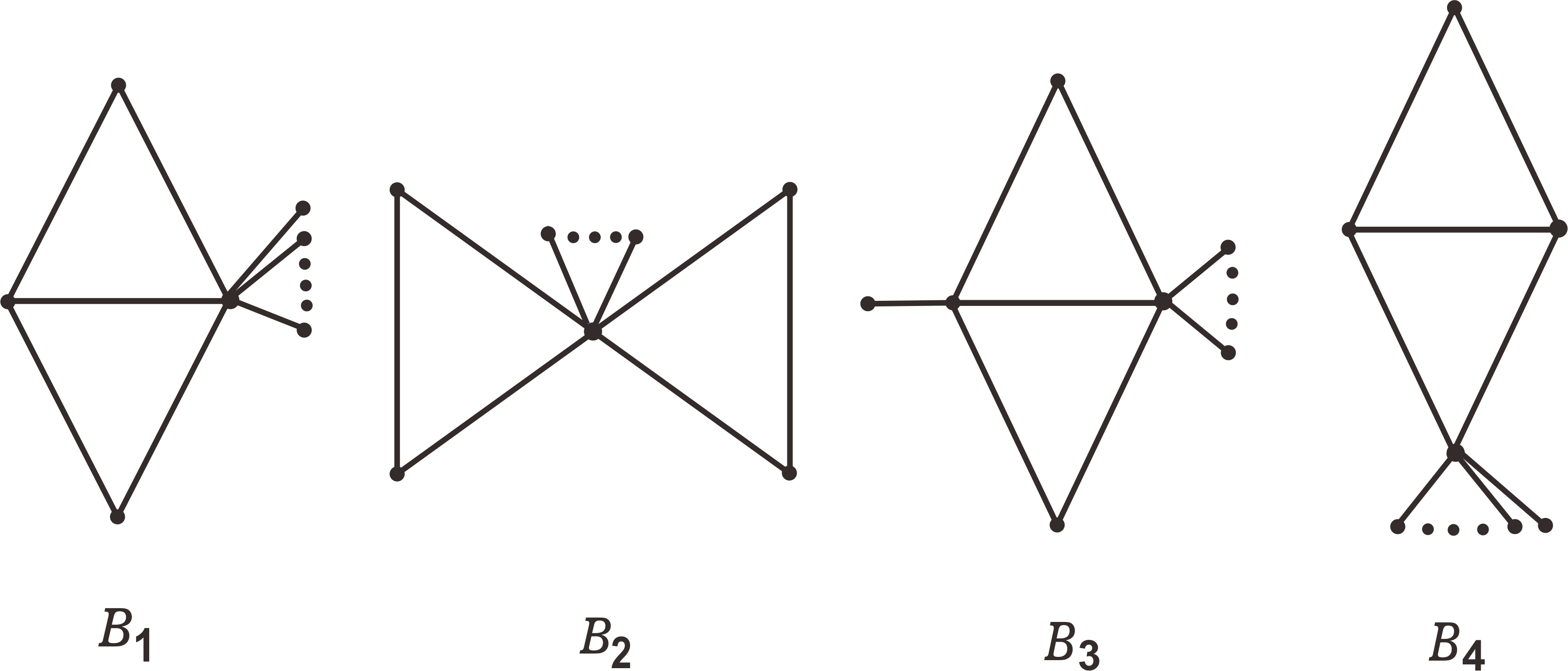}
	\caption{\small  The bicyclic graphs $B_{1}$, $B_{2}$, $B_{3}$ and $B_{4}$ of order $n\geq7$.}
	\label{fig3}
\end{figure}

In the following, we firstly consider the unique bicyclic graph with the largest spectral radius of weighted adjacency matrices with property $P^{\ast}$

\begin{theorem}\label{theorem 5.1} If $G$ is a bicyclic graph of order $n\geq6$ and is not isomorphic to the bicyclic
	graphs $B_{1}$, $B_{2}$, $B_{3}$ and $B_{4}$ as shown in Figure \ref{fig3}, then we have $\displaystyle
	\rho(A_{f}(G))<\rho(A_{f}(B_{3}))$.
\end{theorem}
\begin{proof}
	We consider the following two cases.
	
	$\textbf{Case 1.}$ $G \in \mathcal{B}_{1}(n)$.
	
	$\textbf{Subcase 1.1.}$ $\hat{G}\in\mathcal{P}$ such that $l=2, q=t=3$.
	
	$B_{5}$ is the bicyclic graph of order $n\geq6$ as shown in Figure \ref{fig1054}, where $T_{i}$ is a tree and has a unique common vertex $v_{i}$ with $P(3,2,3)$ in $B_{5}$ for $1\leq i\leq4$.  Without loss of generality, let $v(T_{i})$ be the order of $T_{i}$ such that $v(T_{1})\geq v(T_{2})\geq1$ and $v(T_{3})\geq v(T_{4})\geq1$. If $\hat{G}\in\mathcal{P}$ such that $l=2, q=t=3$, then $G$ has the same structure as $B_{5}$.
	\begin{figure}[H]
		\centering
		\includegraphics[width=3cm]{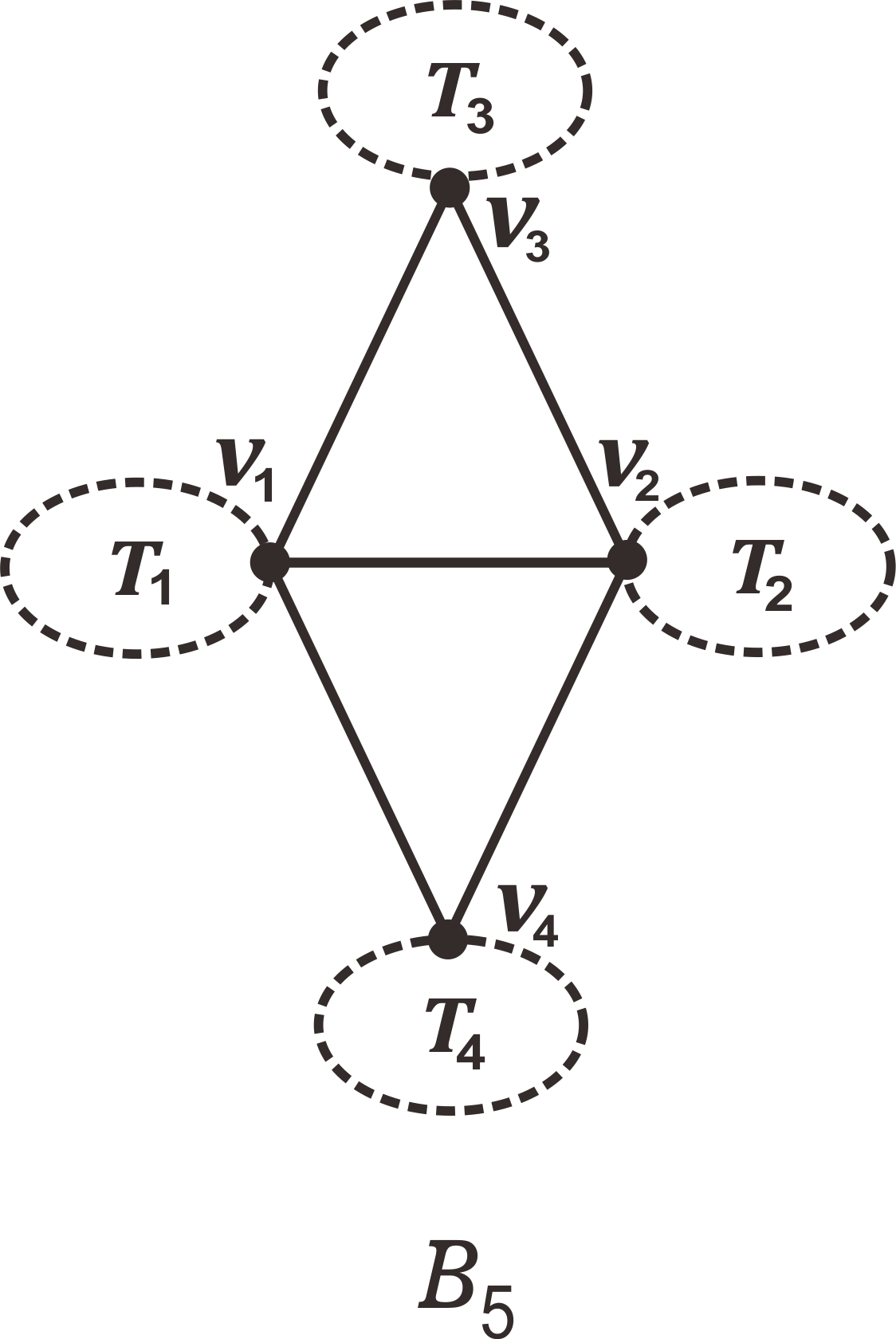}
		\caption{\small  The bicyclic graph $B_{5}$.}
		\label{fig1054}
	\end{figure}
	
	$\textbf{Subcase 1.1.1.}$ $v(T_{1})\geq2$.
	
	If there are non-pendent vertices $w_{1},\ldots,w_{s}$ ($1\leq s\leq\lfloor\frac{n-4}{2}\rfloor$) adjacent to $v_{1}$ in $T_{1}$, then we use Kelmans operations on the vertices $v_{2}$ and $w_{1},\ldots,w_{s}$, $v_{2}$ and $v_{3}$,  $v_{2}$ and $v_{4}$, respectively. Otherwise, we only use Kelmans operations on the vertices $v_{2}$ and $v_{3}$,  $v_{2}$ and $v_{4}$. We denote the obtained graph by $G'$ (see Figure \ref{fig181}). Because $G\ncong B_{1}$, we have $v(T''_{2})\geq2$.
	
	If $G'\cong B_{3}$, then by Theorem \ref{theorem 2.8}, we get $\displaystyle
	\rho(A_{f}(G))<\rho(A_{f}(B_{3}))$.
	
	Otherwise, by Theorems \ref{theorem 2.8}, \ref{theorem 2.9} and \ref{theorem 2.10}, we obtain $\displaystyle
	\rho(A_{f}(G))<\rho(A_{f}(B_{3}))$.
	
	\begin{figure}[H]
		\centering
		\includegraphics[width=7cm]{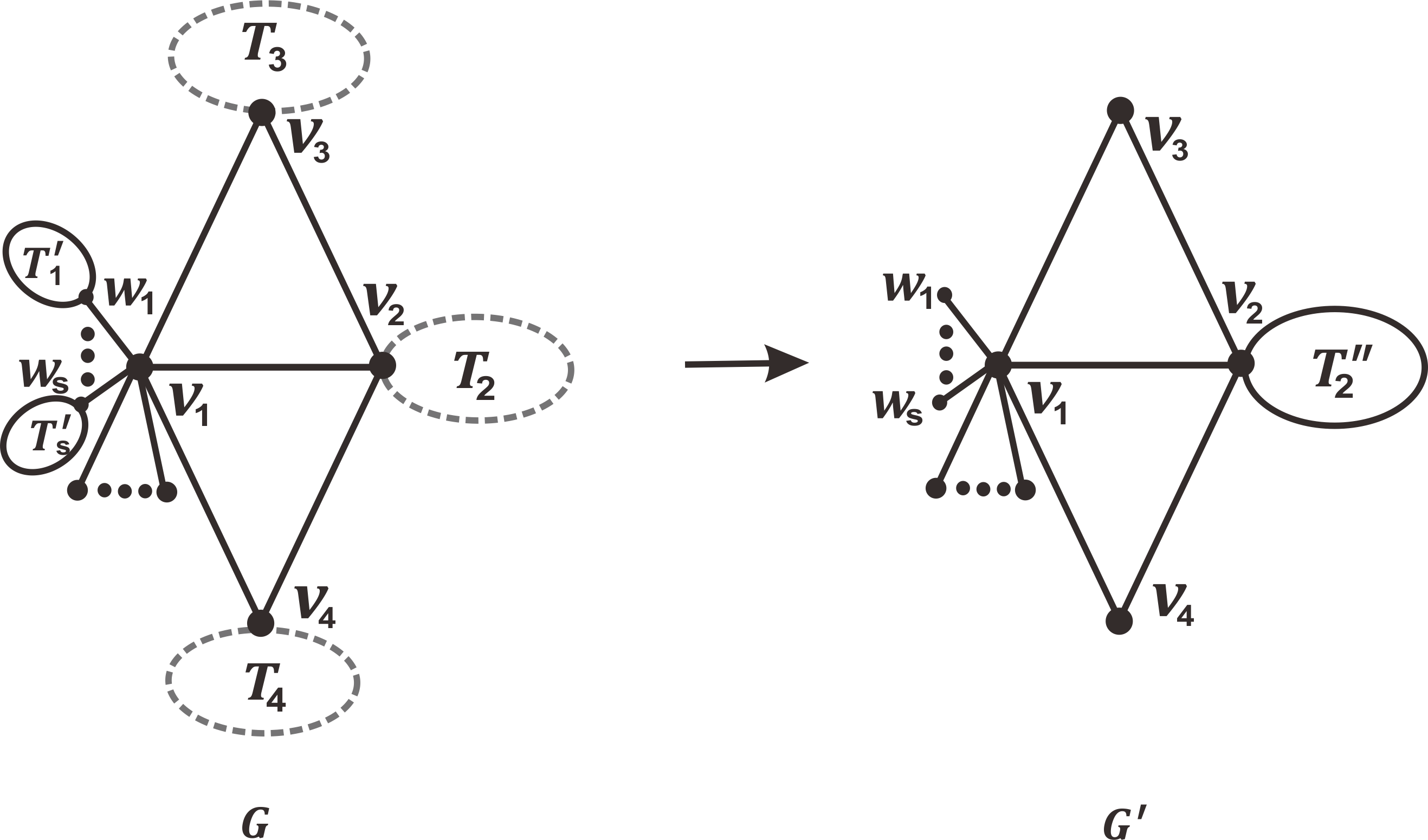}
		\caption{\small  The Kelmans operations for Subcase 1.1.1.}
		\label{fig181}
	\end{figure}
	
	$\textbf{Subcase 1.1.2.}$ $v(T_{1})=1$.
	
	As $v(T_{1})\geq v(T_{2})$, we have $v(T_{1})=v(T_{2})=1$. If there are non-pendent vertices $w_{1},\ldots,w_{s}$ ($1\leq s\leq\lfloor\frac{n-4}{2}\rfloor$) adjacent to $v_{3}$ in $T_{3}$, then we use Kelmans operations on the vertices $v_{1}$ and $w_{1},\ldots,w_{s}$, $v_{1}$ and $v_{4}$,  $v_{2}$ and $v_{3}$, respectively. Otherwise, we only use Kelmans operations on the vertices $v_{1}$ and $v_{4}$,  $v_{2}$ and $v_{3}$. We denote the obtained graph by $G'$ (see Figure \ref{fig182}). Since $G\ncong B_{4}$, then $v(T''_{1})\geq2$.
	
	If $G'\cong B_{3}$, then by Theorem \ref{theorem 2.8}, we have $\displaystyle
	\rho(A_{f}(G))<\rho(A_{f}(B_{3}))$. Otherwise, according to Theorems \ref{theorem 2.8}, \ref{theorem 2.9} and \ref{theorem 2.10}, we obtain $\displaystyle
	\rho(A_{f}(G))<\rho(A_{f}(B_{3}))$.
	
	\begin{figure}[H]
		\centering
		\includegraphics[width=7cm]{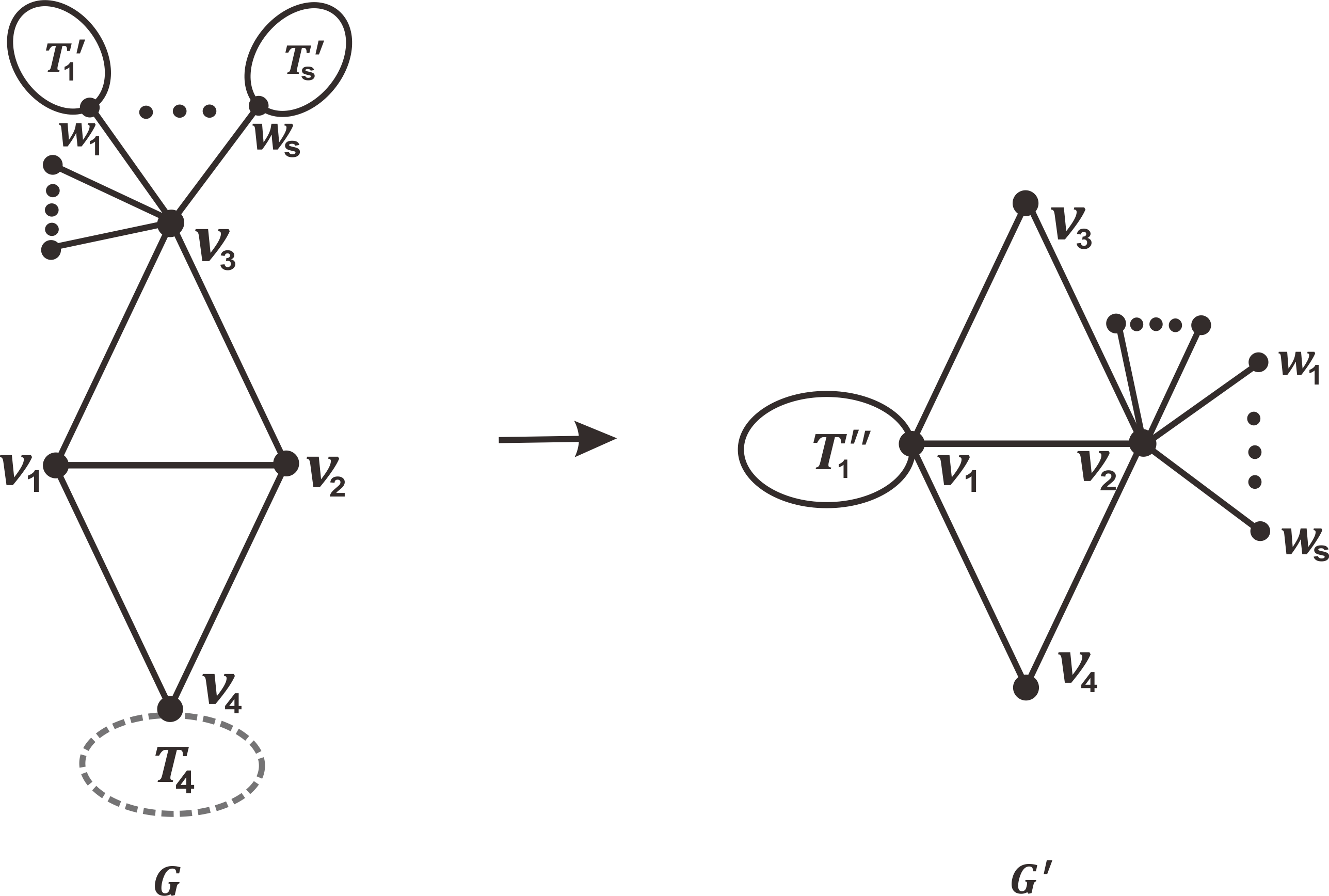}
		\caption{\small  The Kelmans operations for Subcase 1.1.2.}
		\label{fig182}
	\end{figure}
	
	$\textbf{Subcase 1.2.}$  $\hat{G}\in\mathcal{P}$ such that $l=2, q>3$.
	
	Let $d_{z_{1}}+d_{v_{1}}\leq d_{z_{2}}+d_{v_{q-2}}$. We firstly use Kelmans operations on the vertices $v_{1}$ and $z_{1}$, then $z_{2}$ and $v_{q-2}$, $\ldots$, $v_{2}$, $u_{t-2}$, $\ldots$, $u_{1}$ in order to obtain a resulting graph $G'$. Because $d_{z_{1}}+d_{v_{1}}\leq d_{z_{2}}+d_{v_{q-2}}$, then $G'\ncong B_{1}$.
	
	If $G'\cong B_{3}$, according to Theorem \ref{theorem 2.8}, we have $\displaystyle
	\rho(A_{f}(G))<\rho(A_{f}(B_{3}))$.
	Otherwise, by Theorems \ref{theorem 2.8}, \ref{theorem 2.9} and \ref{theorem 2.10}, we obtain $\displaystyle
	\rho(A_{f}(G))<\rho(A_{f}(B_{3}))$.

	$\textbf{Subcase 1.3.}$  $\hat{G}\in\mathcal{P}$ such that $l\geq3$.
	
	We firstly use Kelmans operations on the vertices $z_{l}$ and $v_{q-l}$, $\ldots$, $v_{2}$, $u_{t-l}$, $\ldots$, $u_{2}$ (the vertices $v_{q-l}$, $\ldots$, $v_{2}$, $u_{t-l}$, $\ldots$, $u_{2}$ are not necessarily all exist, we use Kelmans operations on the existing vertices) in order. Then we use Kelmans operations on the vertices $z_{3}$ and $z_{3}$, $\ldots$, $z_{l}$ in order.
	Suppose the resulting graph is $G'$. Then $G'$ has the same structure as $B_{6}$ (see Figure \ref{fig12117} (left)). Afterward, we use Kelmans operations on the vertices $u_{1}$ and $z_{2}$, then $z_{2}$ and $v_{1}$ to obtain a resulting graph that has the same structure as $B_{7}$ (see Figure \ref{fig12117} (center)). Without loss of generality, we assume that $v(T_{2})\geq v(T_{3})\geq 1$. Next, we use Kelmans operation on the vertices $z_{2}$ and $z_{3}$ to obtain a new graph $G''$ that has the same structure as $B_{8}$ (see Figure \ref{fig12117} (right)). Clearly, $G''$ is not isomorphic to $B_{1}$ and $B_{4}$.
	
	If $G''\cong B_{3}$, then by Theorem \ref{theorem 2.8}, we have $\displaystyle
	\rho(A_{f}(G))<\rho(A_{f}(B_{3}))$. Otherwise, by Subcase 1.1, we get $\displaystyle
	\rho(A_{f}(G))<\rho(A_{f}(B_{3}))$.
	
	\begin{figure}[H]
		\centering
		\includegraphics[width=12.5cm]{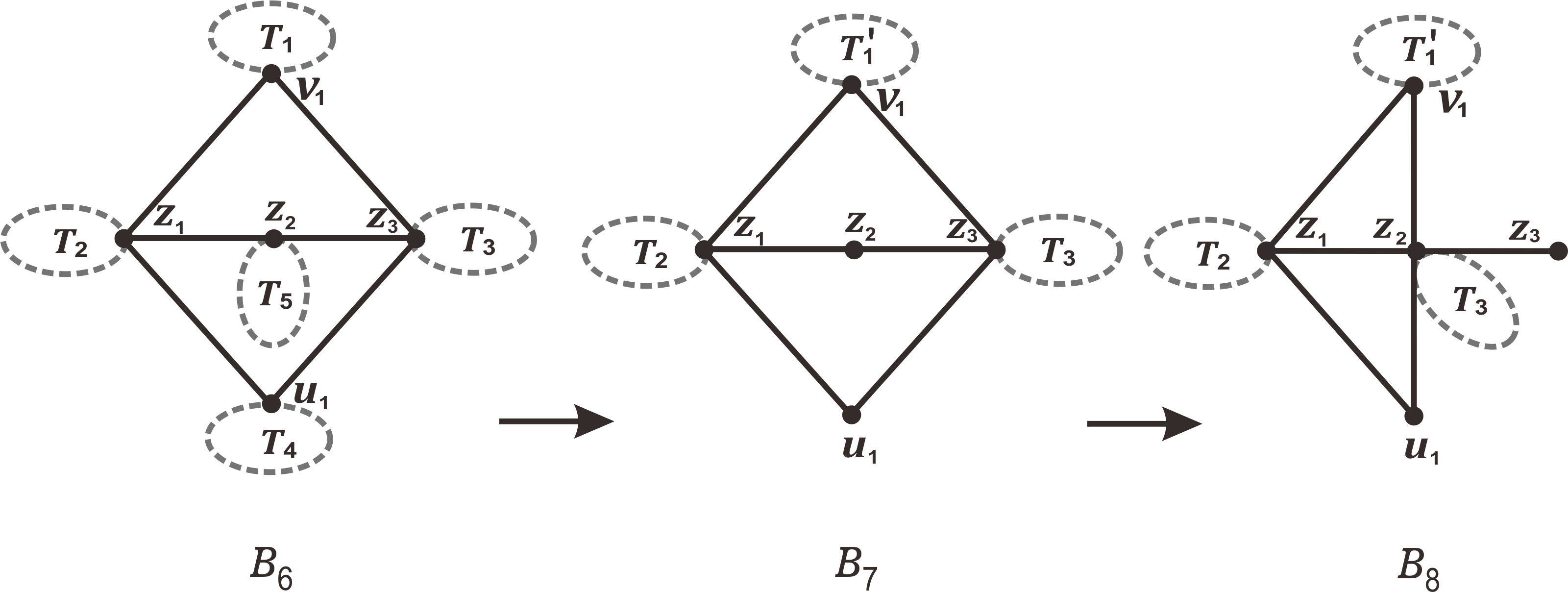}
		\caption{\small  The bicyclic graphs $B_{6}$, $B_{7}$ and $B_{8}$.}
		\label{fig12117}
	\end{figure}
	
	$\textbf{Case 2.}$ $G \in \mathcal{B}_{2}(n)$.
	
	$\textbf{Subcase 2.1.}$  $\hat{G}\in\mathcal{B}$ such that $l=1$.
	
	Firstly, we use Kelmans operations on the vertices $v_{1}$ and $v_{2}$, $\ldots$, $v_{q-1}$, $u_{1}$ and $\ldots$, $u_{t-1}$ in order.
	Suppose the resulting graph is $G'$. We obtain $G'$ has the same structure as $B_{9}$ (see Figure \ref{fig12115} (left)).
	Then we use Kelmans operation on the vertices $v_{1}$ and $u_{t-1}$ to obtain a new graph $G''$ as shown in Figure \ref{fig12115}. It is clear that $G''\ncong B_{4}$. Because $G\ncong B_{2}$, we have $G''\ncong B_{1}$.
	
	If $G''\cong B_{3}$, then by Theorem \ref{theorem 2.8}, we have $\displaystyle
	\rho(A_{f}(G))<\rho(A_{f}(B_{3}))$. Otherwise, according to Subcase 1.1, we get $\displaystyle
	\rho(A_{f}(G))<\rho(A_{f}(B_{3}))$.
	\begin{figure}[H]
		\centering
		\includegraphics[width=10.5cm]{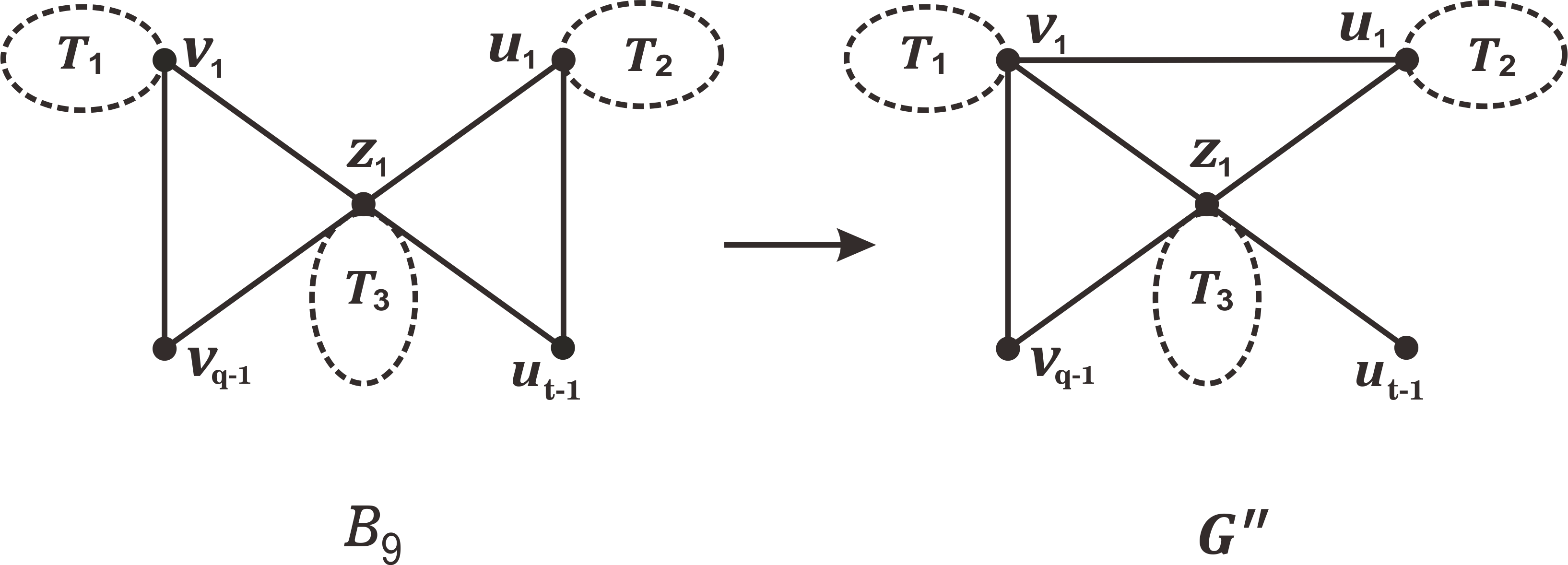}
		\caption{\small  The bicyclic graph $B_{9}$ and the Kelmans operation for Subcase 2.1.}
		\label{fig12115}
	\end{figure}

	$\textbf{Subcase 2.2.}$  $\hat{G}\in\mathcal{B}$ such that $l\geq2$.
	
	Firstly, we use Kelmans operations on the vertices $v_{1}$ and $v_{2}$, $v_{3}$, $\ldots$, $v_{q-1}$, the vertices $u_{1}$ and $u_{2}$, $u_{3}$, $\ldots$, $u_{t-1}$, the vertices $z_{2}$ and $z_{2}$, $z_{3}$, $\ldots$, $z_{l-1}$, $z_{l}$ of $G$ in order. Suppose the resulting graph is $G'$. Then we use Kelmans operation on the vertices $v_{1}$ and $z_{2}$ to obtain a new graph $G''$ as shown in Figure \ref{fig12113}. We have $G''\ncong B_{2}$. Analogously, we get $\displaystyle
	\rho(A_{f}(G))<\rho(A_{f}(B_{3}))$.
	
	This proof is completed.
\end{proof}
\begin{figure}[H]
	\centering
	\includegraphics[width=11cm]{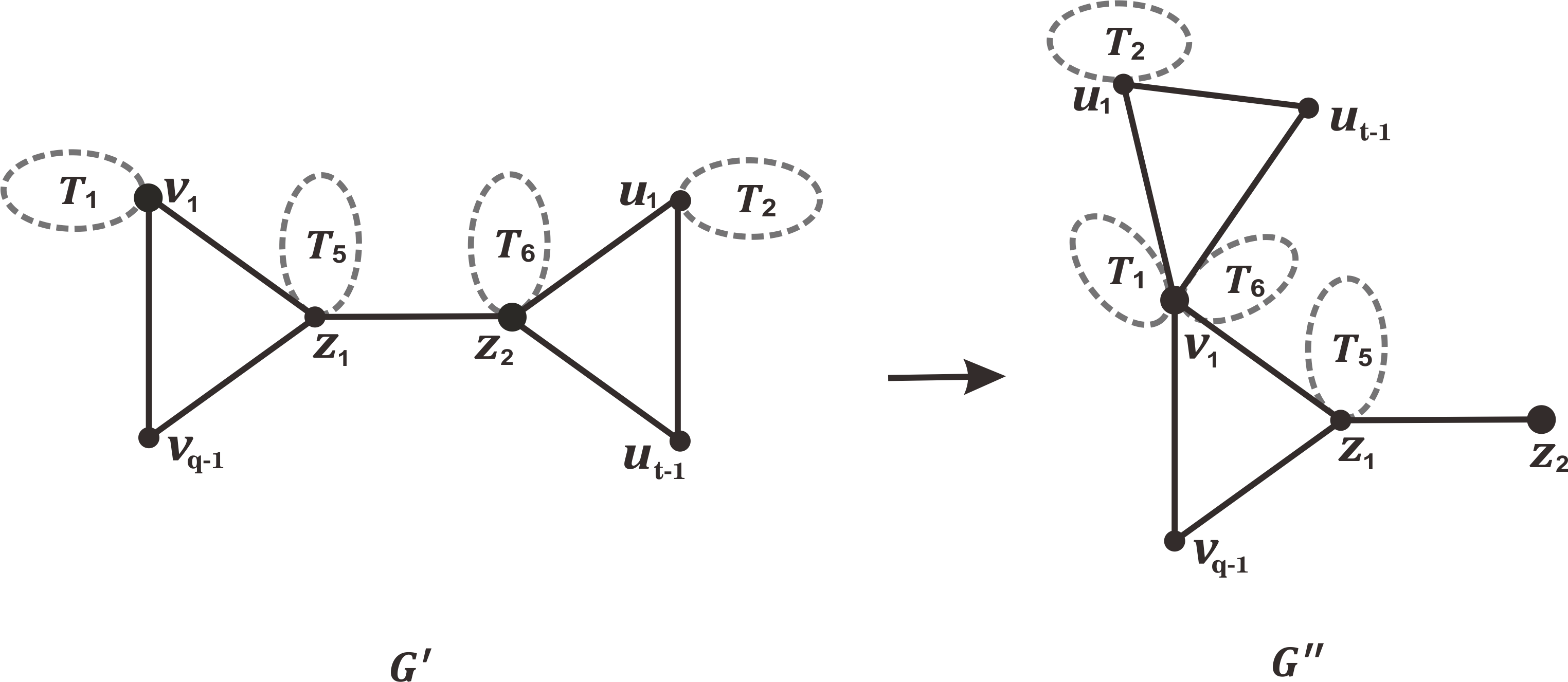}
	\caption{\small  The Kelmans operation for Subcase 2.2.}
	\label{fig12113}
\end{figure}

\begin{theorem}\label{theorem 5.2} Among all bicyclic graphs of order $n\geq6$, $B_{1}$ is the unique bicyclic graph with the largest spectral radius of weighted adjacency matrices with property $P^{\ast}$.
\end{theorem}
\begin{proof}
	Among the bicyclic graphs $B_{1}$, $B_{2}$, $B_{3}$ and $B_{4}$, according to Theorem \ref{theorem 2.8}, we get $B_{1}$ with the largest spectral radius of weighted adjacency matrices with property $P^{\ast}$. Incorporate with Theorem \ref{theorem 5.1}, we have thus proved the theorem.
\end{proof}

Next, we will think over the unique bicyclic graphs with the first three maximum $p$-Sombor spectral radii.

\begin{lemma}\label{lemma 5.1}
	$B_{3}$, $B_{4}$ are the bicyclic graphs of order $n\geq6$ as shown in Figure \ref{fig3}. Then we have $\displaystyle \rho(\textbf{S}_{\textbf{p}}(B_{4}))<\rho(\textbf{S}_{\textbf{p}}(B_{3}))$.
\end{lemma}
\begin{proof} In the proof, we meditate the following two cases.
	
	$\textbf{Case 1.}$ $n\geq10$.
	
	Let the vertices of $B_{4}$ be denoted as Figure \ref{fig1055} (left).
	\begin{figure}[H]
		\centering
		\includegraphics[width=8cm]{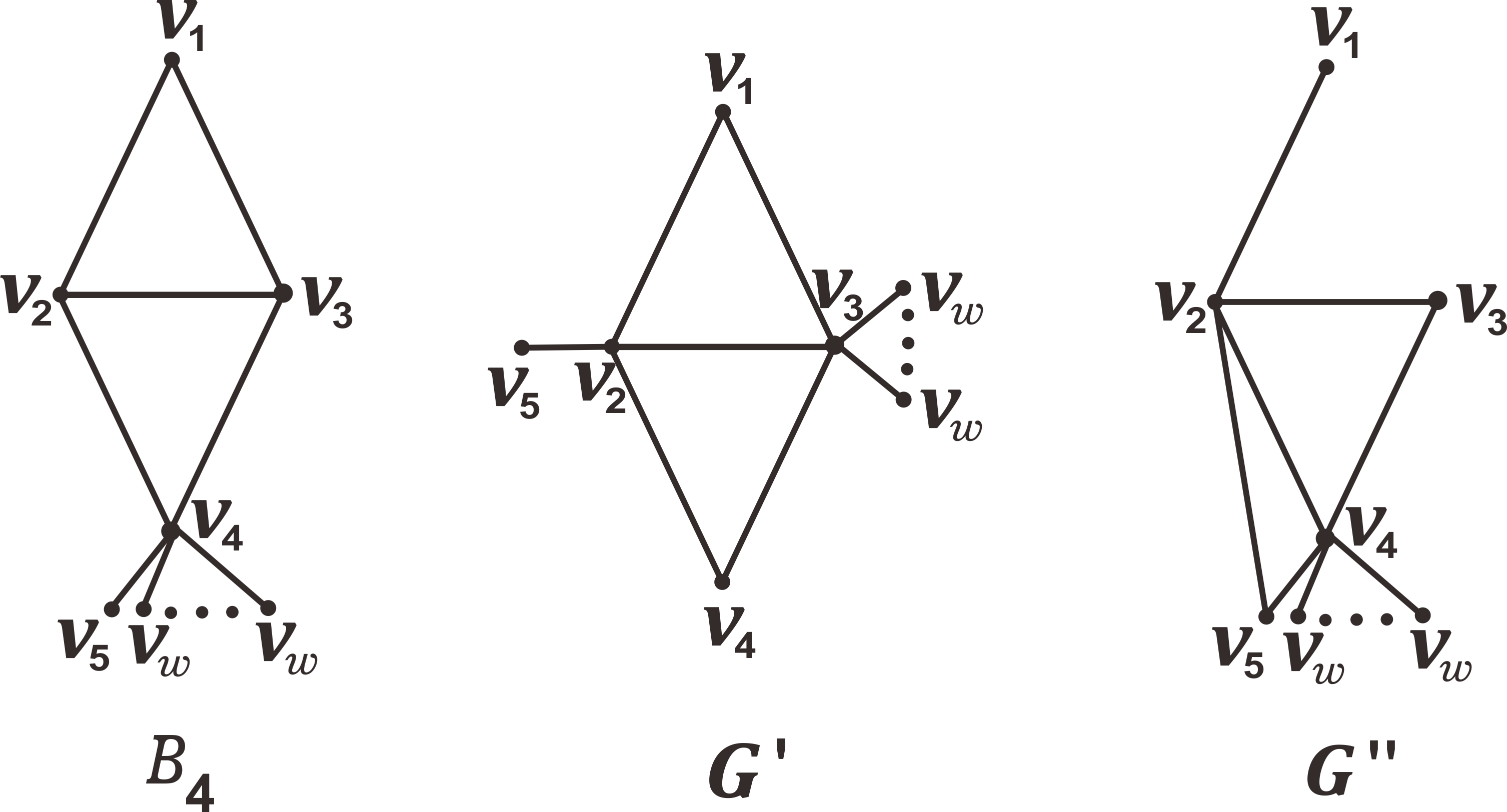}
		\caption{\small  The bicyclic graphs $B_{4}$, $G'$ and $G''$ for Lemma \ref{lemma 5.1}.}
		\label{fig1055}
	\end{figure}
	
	Suppose that $\textbf{x}$ is the principal eigenvector of $B_{4}$.
	Thanks to $\textbf{S}_{\textbf{p}}(B_{4})\textbf{x}=\rho(\textbf{S}_{\textbf{p}}(B_{4}))\textbf{x}$, the entries corresponding to the pendent vertices $v_{w}$ and $v_{5}$ in $\textbf{x}$ are equal. We get
	$$\begin{aligned}
		&\rho(\textbf{S}_{\textbf{p}}(B_{4}))x_{1}=2f(2,3,p)x_{2},\\
		& \rho(\textbf{S}_{\textbf{p}}(B_{4}))x_{2}=f(2,3,p)x_{1}+f(3,3,p)x_{2}+f(3,n-2,p)x_{4},\\
		& \rho(\textbf{S}_{\textbf{p}}(B_{4}))x_{4}=2f(3,n-2,p)x_{2}+(n-4)f(1,n-2,p)x_{5},\\
		& \rho(\textbf{S}_{\textbf{p}}(B_{4}))x_{5}=f(1,n-2,p)x_{4}.
	\end{aligned}$$
	We consider the following two cases.
	
	\textbf{Subcase 1.1.} $\displaystyle x_{1}>x_{5}$.
	
	Since $\displaystyle f(1,n-2,p)>2f(2,3,p)$ for $n\geq10$, we obtain if $x_{1}>x_{5}$, then $x_{2}>x_{4}$. We have
	$x_{2}=x_{3}>x_{4}>x_{1}>x_{5}$. Replace edge $v_{4}v_{5}$ in $B_{4}$ by a new edge
	$v_{2}v_{5}$ and $v_{4}v_{w}$ by a new edge $v_{3}v_{w}$ for $n-5$ pendent vertices $v_{w}\in N(v_{4})$ to obtain a new
	graph $G'\cong B_{3}$ as shown in Figure \ref{fig1055} (centre). Then we have
	$$\begin{aligned}
		&\quad
		\textbf{x}^{\top}\textbf{S}_{\textbf{p}}(G')\textbf{x}-\textbf{x}^{\top}\textbf{S}_{\textbf{p}}(B_{4})\textbf{x}=\\
		& 2\big[f(2,4,p)-f(2,3,p)\big]x_{1}x_{2}+2\big[f(2,n-2,p)-f(2,3,p)\big]x_{1}x_{3}\\
		+&2\big[f(1,4,p)x_{2}x_{5}-f(1,n-2,p)x_{4}x_{5}\big]\\
		+&2\big[f(4,n-2,p)-f(3,3,p)\big]x_{2}x_{3}+2\big[f(2,4,p)-f(3,n-2,p)\big]x_{2}x_{4}\\
		+&2\big[f(2,n-2,p)-f(3,n-2,p)\big]x_{3}x_{4}\\
		+&2(n-5)\big[f(1,n-2,p)x_{3}x_{w}-f(1,n-2,p)x_{4}x_{w}\big]\\
		>&2\big[f(2,4,p)-f(2,3,p)+f(2,n-2,p)-f(2,3,p)+f(1,4,p)-f(1,n-2,p)\big]x_{2}x_{5}\\
		+&2\big[f(4,n-2,p)-f(3,3,p)+f(2,4,p)-f(3,n-2,p)+f(2,n-2,p)\\
		-&f(3,n-2,p)\big]x_{2}x_{3}+2(n-5)\big[f(1,n-2,p)x_{3}x_{w}-f(1,n-2,p)x_{4}x_{w}\big]>0.
	\end{aligned}$$
	Thus we have $\displaystyle \rho(\textbf{S}_{\textbf{p}}(B_{4}))<\rho(\textbf{S}_{\textbf{p}}(B_{3}))$.
	
	\textbf{Subcase 1.2.} $\displaystyle x_{1}\leq x_{5}$.
	
	We replace edge $v_{1}v_{3}$ in $B_{4}$ by a new
	edge $v_{2}v_{5}$ to obtain a new graph $G''\cong B_{3}$ as shown in Figure \ref{fig1055} (right). Then we have
	$$\begin{aligned}
		&\quad
		\textbf{x}^{\top}\textbf{S}_{\textbf{p}}(G'')\textbf{x}-\textbf{x}^{\top}\textbf{S}_{\textbf{p}}(B_{4})\textbf{x}=\\
		& 2\big[f(1,4,p)-f(2,3,p)\big]x_{1}x_{2}+2\big[f(2,4,p)-f(3,3,p)\big]x_{2}x_{3}\\
		+&2\big[f(2,n-2,p)-f(3,n-2,p)\big]x_{3}x_{4}+2\big[f(4,n-2,p)-f(3,n-2,p)\big]x_{2}x_{4}\\
		+&2\big[f(2,n-2,p)-f(1,n-2,p)\big]x_{4}x_{5}+2\big[f(2,4,p)x_{2}x_{5}-f(2,3,p)x_{1}x_{3}\big]>0.
	\end{aligned}$$
	Thus we obtain $\displaystyle \rho(\textbf{S}_{\textbf{p}}(B_{4}))<\rho(\textbf{S}_{\textbf{p}}(B_{3}))$ for $n\geq10$.
	
	$\textbf{Case 2.}$ $6\leq n\leq9$.
	
	By Lemma \ref{lemma 2.1} and under the aid of MATLAB, we also can get $\displaystyle \rho(\textbf{S}_{\textbf{p}}(B_{4}))<\rho(\textbf{S}_{\textbf{p}}(B_{3}))$ for $6\leq n\leq9$.
	The proof of this result is quite similar to that given for Case 2 of Lemma \ref{lemma 4.1}.
	
	We have thus proved the lemma.
\end{proof}

\begin{lemma}\label{lemma 5.2}
	$B_{3}$ and $B_{2}$ are the bicyclic graphs of order $n\geq6$ as shown in Figure \ref{fig3}. Then we have $\displaystyle
	\rho(\textbf{S}_{\textbf{p}}(B_{3}))<\rho(\textbf{S}_{\textbf{p}}(B_{2}))$.
\end{lemma}
\begin{proof} The proof of the lemma is divided into the following two cases.
	
	$\textbf{Case 1.}$ $n\geq18$.
	
	Let the vertices of $B_{2}$ be denoted as Figure \ref{fig231201}.
	\begin{figure}[H]
		\centering
		\includegraphics[width=3cm]{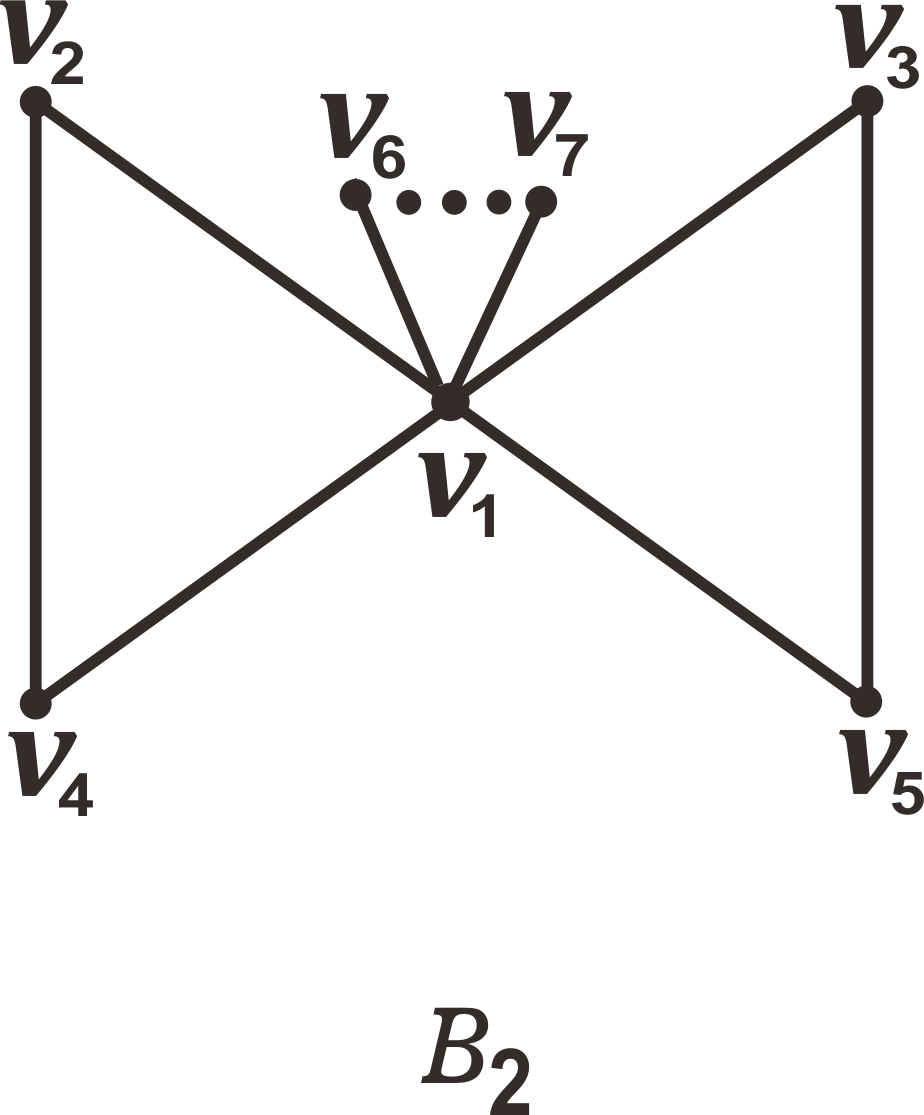}
		\caption{\small  The bicyclic graph $B_{2}$ of order $n$.}
		\label{fig231201}
	\end{figure}
	We denote the quotient matrix of the equitable partition $\{\{v_{1}\}, \{v_{2},v_{3},v_{4},v_{5}\},$
	
	\noindent $\{v_{6},v_{7},\ldots,v_{n}\}\}$ of $\textbf{S}_{\textbf{p}}(B_{2})$ as
	\begin{equation*}
		Q=\begin{pmatrix}
			0 & 4f(2,n-1,p) & (n-5)f(1,n-1,p)\\
			f(2,n-1,p) & f(2,2,p) & 0 \\
			f(1,n-1,p) & 0 & 0\\
		\end{pmatrix}.
	\end{equation*}
	Let $\displaystyle P(x, B_{2})$ be the characteristic polynomial of the quotient matrix $Q$.
	We can obtain
	\begin{eqnarray*}
		P(x, B_{2})&=&x^{3}-f(2,2,p)x^{2}-\big[4f^{2}(2,n-1,p)+(n-5)f^{2}(1,n-1,p)\big]x\\
		&&+(n-5)f(2,2,p)f^{2}(1,n-1,p).
	\end{eqnarray*}
	
	Similarly, let $\displaystyle P'(x, B_{3})$ be the characteristic polynomial of the quotient matrix of an equitable partition of the adjacency matrix
	$A(B_{3})$, we have
	
	\begin{center}
		$\displaystyle P'(x, B_{3})=x^{4}-(n+1)x^{2}-4x+3n-13$.
	\end{center}
	
	By calculation, we obtain $\displaystyle \rho(B_{3})<\sqrt{n}$ for $n\geq18$. Thus $\displaystyle
	\rho(\textbf{S}_{\textbf{p}}(B_{3}))<\sqrt{n}f(4,n-2,p)$.
	It is easy to prove
	$$\begin{aligned}
		&\displaystyle  P(\sqrt{n}f(4,n-2,p), B_{2})\\
		=&n\sqrt{n}f^{3}(4,n-2,p)-nf(2,2,p)f^{2}(4,n-2,p)\\
		-&\sqrt{n}f(4,n-2,p)\big[4f^{2}(2,n-1,p)+(n-5)f^{2}(1,n-1,p)\big]+(n-5)f(2,2,p)f^{2}(1,n-1,p)\\
		=&\sqrt{n}f(4,n-2,p)\big[nf^{2}(4,n-2,p)-4f^{2}(2,n-1,p)-(n-5)f^{2}(1,n-1,p)\big]\\
		+&f(2,2,p)\big[(n-5)f^{2}(1,n-1,p)-nf^{2}(4,n-2,p)\big]\\
		\leq&\sqrt{n}f(4,n-2,p)\big\{n[16+(n-2)^{2}]-4(n-1)^{2}-(n-5)(n-1)^{2}\big\}\\
		+&f(2,2,p)\big\{(n-5)[1+(n-1)^{2}]-n(n-2)^{2}\big\} \ (\textrm{by} \ \textrm{Lemma} \ \ref{lemma 2.1})\\
		=&\sqrt{n}f(4,n-2,p)(-n^{2}+17n+1)+f(2,2,p)(-3n^{2}+8n-10)<0
	\end{aligned}$$
	for $n\geq18$.
	
	Then $\displaystyle
	\rho(\textbf{S}_{\textbf{p}}(B_{3}))<\rho(\textbf{S}_{\textbf{p}}(B_{2}))$ for $n\geq18$.
	
	$\textbf{Case 2.}$  $6\leq n\leq17$.
	
	Analogously, with the aid of MATLAB and by Lemma \ref{lemma 2.1}, we have $\displaystyle \rho(\textbf{S}_{\textbf{p}}(B_{3}))<\rho(\textbf{S}_{\textbf{p}}(B_{2}))$ for $6\leq n\leq17$. The proof of this result is quite similar to that given for Case 2 of Lemma \ref{lemma 4.1}.
	
	This completes the proof.
\end{proof}

\begin{theorem}\label{theorem 5.2}
	Among all bicyclic graphs of order $n\geq6$, $B_{1}$, $B_{2}$ and $B_{3}$ (see Figure \ref{fig3}) are,
	respectively, the unique bicyclic graphs with the first three maximum $p$-Sombor spectral radii.
\end{theorem}
\begin{proof} By Theorem \ref{theorem 2.8}, we have $\displaystyle
	\rho(\textbf{S}_{\textbf{p}}(B_{2}))<\rho(\textbf{S}_{\textbf{p}}(B_{1}))$.
	Then according to Theorem \ref{theorem 5.1} and Lemmas \ref{lemma 5.1} and \ref{lemma 5.2}, it is easy to see $B_{1}$,
	$B_{2}$ and $B_{3}$ are, respectively, the unique bicyclic graphs with the first three
	maximum $p$-Sombor spectral radii for $n\geq6$.
\end{proof}

\noindent  \textbf{Remark 5.1.} For $n=5$, we have $B_{1}\cong B_{3}$. Under the help of MATLAB and by Theorem \ref{theorem 2.8}, we get $B_{1}$ is the unique bicyclic graph with the largest $p$-Sombor spectral radius, the bicyclic graphs with the secondly and thirdly largest $p$-Sombor spectral radii depend on the choice of parameter $p$ and they are $B_{2}$ and $B_{4}$.

\section{Conclusions}
\noindent
In this paper, among all trees of order $n\geq6$, we obtain $S_{n}$, $S_{2,n-2}$, $S_{3,n-3}$ are, respectively, the unique trees with the first three maximum spectral radii for the weighted adjacency matrix with property $P^{\ast}$. In addition, we get
\begin{center}
	$\rho(A_{f}(P_{4}))=\rho(A_{f}(S_{2,2}))<\rho(A_{f}(S_{4}))$ for $n=4$
\end{center}
\noindent and
\begin{center}
	$\rho(A_{f}(P_{5}))<\rho(A_{f}(S_{2,3}))<\rho(A_{f}(S_{5}))$ for $n=5$.
\end{center}
The result also holds for the $p$-Sombor matrix with $p\geq1$.

Among all unicyclic graphs of order $n\geq7$, we get $S_{n}+e$, $U_{1}$, $U_{2}$, $U_{3}$ and $U_{4}$ are, respectively, the unique unicyclic graphs with the first five maximum $p$-Sombor spectral radii for $p\geq2$.
For $n=6$, we
get $S_{6}+e$, $U_{1}$ and $U_{2}$ are, respectively, the unique unicyclic
graphs with the first three maximum $p$-Sombor spectral radii and the unicyclic
graph with the fourthly largest $p$-Sombor spectral radius depends on the choice of parameter $p$.
Furthermore, we obtain
\begin{center}
	$\rho(\textbf{S}_{\textbf{p}}(C_{5}))<\rho(\textbf{S}_{\textbf{p}}(U_{3}))<\rho(\textbf{S}_{\textbf{p}}(U_{2}))<\rho(\textbf{S}_{\textbf{p}}(U_{1}))<\rho(\textbf{S}_{\textbf{p}}(S_{5}+e))$
\end{center}
\noindent for $n=5$.

Moreover, among all bicyclic graphs of order $n\geq6$, we obtain $B_{1}$ is the unique bicyclic graph with the largest spectral radius of the weighted adjacency matrices with property $P^{\ast}$. In addition, $B_{1}$, $B_{2}$ and $B_{3}$ are, respectively, the unique bicyclic graphs with the first three maximum $p$-Sombor spectral radii for $p\geq2$. For $n=5$, we have $B_{1}\cong B_{3}$ is the unique bicyclic graph with the largest $p$-Sombor spectral radius, the bicyclic graphs with the secondly and thirdly largest $p$-Sombor spectral radii depend on the choice of parameter $p$ and they are $B_{2}$ and $B_{4}$.

In Appendix 1, we give an algorithm to obtain the extremal graphs with the first three minimum Sombor spectral radii among all graphs with order $n$ and size $m$. For example, we get $B'(n,1)$, $B'(n,2)$ and $B'(n,3)$ (see Figure \ref{fig1060}) are, respectively, the bicyclic graphs with the first three minimum Sombor spectral radii for $n=5,6,7$, which dose not show much sign of regularity. We find the problem for the first three minimum Sombor spectral radii of these three classes of graphs in Problem \ref{problem 1.1} is hard for us. The authors have obtained $P_{n}$ (resp. $C_{n}$) is the unique tree (resp. unicyclic graph) with the smallest spectral radius for the weighted adjacency matrix with property $P^{\ast}$ in \cite{107}. Thus we are interested in the following problems. Determine the unique bicyclic graph with the smallest spectral radius for the weighted adjacency matrix with property $P^{\ast}$ among all bicyclic graphs of order $n$. Consider the extremal graph with respect to the spectral radius of weighted adjacency matrix with property $P^{\ast}$ in the set of graphs with prescribed degree sequence.

\begin{figure}[H]
	\centering
	\includegraphics[width=9cm]{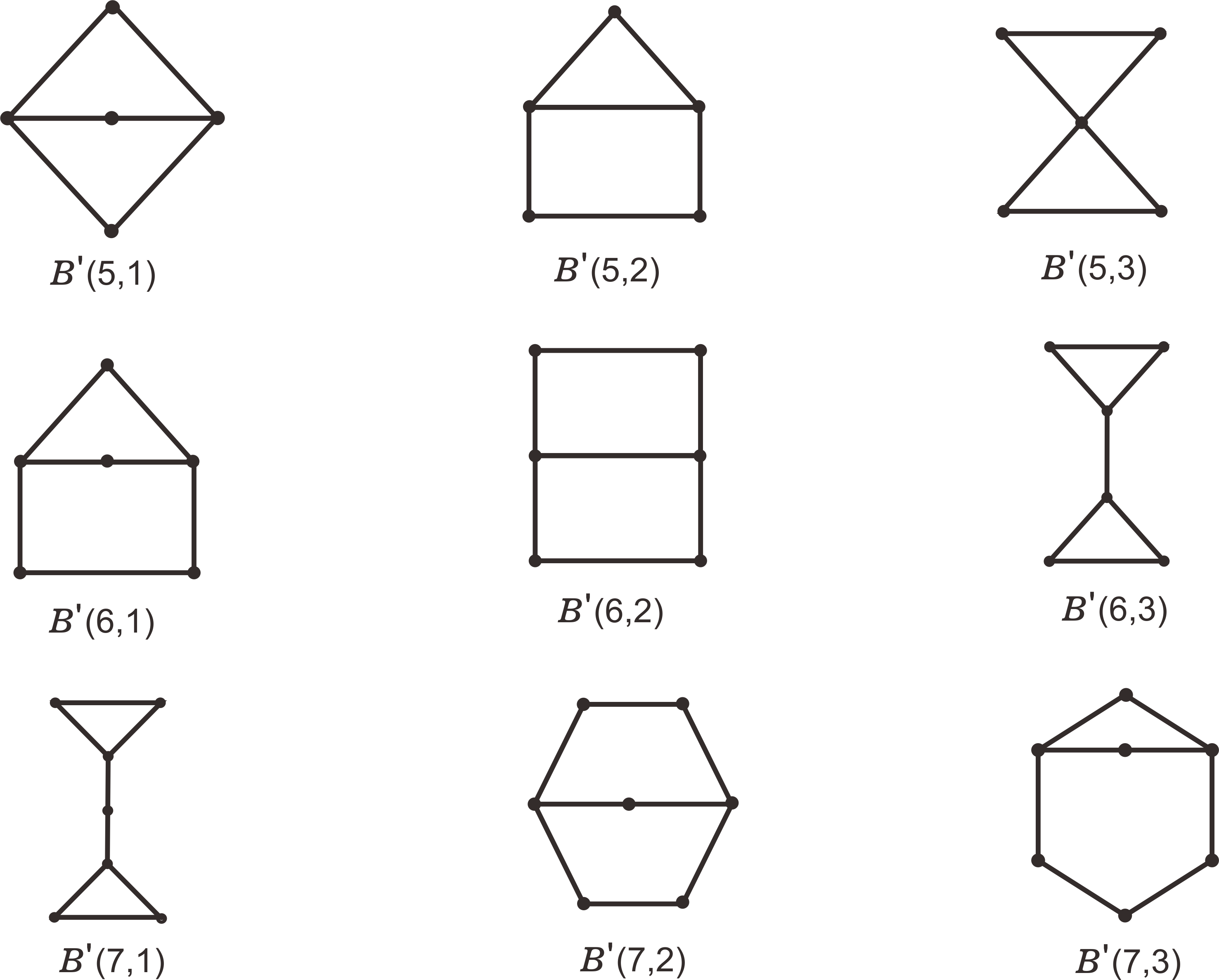}
	\caption{\small  The bicyclic graphs $B'(n,1)$, $B'(n,2)$ and $B'(n,3)$ for $n=5,6,7$.}
	\label{fig1060}
\end{figure}

\section*{Acknowledgements}
\noindent
This work is supported by NSFC (No. 12171402).

\vskip0.5cm

\bibliographystyle{abbrv}
\bibliography{References}

\newpage

\textbf{Appendix 1}\label{Appendix 1}

To find the extremal graphs with the first three minimum Sombor spectral radii among all graphs with order $n$ and size $m$, we present an algorithm which can be achieved by a computer program written in the MATLAB language.
For given $n$ and $m$, it is not difficult to construct recurrently all the graphs with order $n$ and
size $m$ which are non-isomorphic to each other by using graph theory toolbox in MATLAB. In the following algorithm, we assume that all graphs with order $n$ and size $m$ which are non-isomorphic to each other have been produced, and store them as a matrix $M$, where $(2i-1)$-th row and $2i$-th row of $M$ are edge-information-table (that is, EdgeTable in MATLAB) for $i=1,2,\ldots$.

\begin{algorithm}
	\scriptsize
	\SetAlgoLined
	\KwIn{The order $n$ and size $m$.}
	\KwOut{The extremal graphs with the first three minimum Sombor spectral radii among all graphs with order $n$ and
		size $m$. }
	Let $\mathcal{G}=\{G_1,G_2,\ldots,G_t\}$ be the set of graphs with $n$ and size $m$ which are non-isomorphic to each other. Let $S=\emptyset$.\\
	\For{i=1:t}{
		compute $S_p(G_i)$;\\
		compute $\rho(S_p(G_i))$;\\
		set $S=S\cup \rho(S_p(G_i))$.\\	
	}
		\For{s=1:3}{
      find the maximum element $s$ in $S$, that is, the maximum $p$-Sombor spectral radius among all graphs with order $n$ and
      size $m$;\\
      determine the graph $G$ with maximum $p$-Sombor spectral radius in $\mathcal{G}$;\\
      set $S=S\setminus \{s\}$ and $\mathcal{G}=\mathcal{G}\setminus \{G\}$.\\
		}

	\caption{Algorithm to obtain the extremal graphs with the first three minimum Sombor spectral radii among all
		graphs with order $n$ and size $m$.}
\end{algorithm}

\newpage
\begin{verbatim}
Numbergraphs=size(M,1);	
Spectralradius=[];

for k=1:Numbergraphs/2;
    edgetablerow1=M(2*k-1,:);
    edgetablerow2=M(2*k,:);
    Graph=graph(edgetablerow1,edgetablerow2);
    Degreeofvertexu=degree(Graph,edgetablerow1);
    Degreeofvertexv=degree(Graph,edgetablerow2);
    Weights=[(Degreeofvertexu.^(2)+Degreeofvertexv.^(2)).^(1/2)];
    Weightedgraph=graph(edgetablerow1,edgetablerow2,Weights);
    Weightedadjacency=adjacency(Weightedgraph,'weighted');
    Spectrum=eig(Weightedadjacency);
    Spectralradius(k)=max(abs(Spectrum));
end
for s=1:3;
    [Maxspec,Maxspecpposition]=max(Spectralradius);
    Maxgrow1=M(2*Maxspecpposition-1,:);
    Maxgrow2=M(2*Maxspecpposition,:);
    Maxgraph=graph(Maxgrow1,Maxgrow2);
    Maxgraphdegu=degree(Maxgraph,Maxgrow1);
    Maxgraphdegv=degree(Maxgraph,Maxgrow2);
    Maxgraphweights=[(Maxgraphdegu.^(2)+Maxgraphdegv.^(2)).^(1/2)];
    Maxggraphweighted=graph(Maxgrow1,Maxgrow2,Maxgraphweights);
    Spectralradius(Maxspecpposition)=0;
    figure
    plot(Maxggraphweighted,'EdgeLabel',Maxggraphweighted.Edges.Weight);
    hold off
end
\end{verbatim}

\end{document}